\documentclass[12pt,reqno]{amsart}

\usepackage[left=3cm,right=3cm,top=3.5cm,bottom=3.5cm]{geometry}
\usepackage{amssymb}
\usepackage{graphicx}
\usepackage{amscd}
\usepackage[pagebackref]{hyperref}
\usepackage{color}
\usepackage{tabularx}
\usepackage[table]{xcolor}
\usepackage{float}
\usepackage{graphics,amsmath,amssymb}
\usepackage{amsthm}
\usepackage{amsfonts}
\usepackage{latexsym}
\usepackage{epsf}
\usepackage{xifthen}
\usepackage{mathrsfs}
\usepackage{dsfont}
\usepackage{makecell}
\usepackage{subfig}
\usepackage{amsmath}
\allowdisplaybreaks[4]
\usepackage{listings}
\usepackage{etoolbox}
\usepackage{fancyhdr}
\usepackage{pdflscape}
\usepackage[title,toc,titletoc]{appendix}
\usepackage{enumitem}
\usepackage[noadjust]{cite}
\usepackage{tikz}
\usetikzlibrary{automata,positioning,arrows}
\usepackage{young}
\usepackage[object=vectorian]{pgfornament} 
\usepackage{lipsum,tikz}
\usepackage{multirow}
\usepackage[OT2,T1]{fontenc}
\usepackage{mathtools}

\hypersetup{
	colorlinks=true, 
	linktoc=all, 
	linkcolor=blue} 

\numberwithin{equation}{section}

\theoremstyle{theorem}
\newtheorem{theorem}{Theorem}[section]
\newtheorem*{theorem*}{Theorem}

\newtheorem{lemma}[theorem]{Lemma}

\newtheorem{fact}[theorem]{Fact}

\newtheorem{question}[theorem]{Question}

\newtheorem{innercustomgeneric}{\customgenericname}
\providecommand{\customgenericname}{}
\newcommand{\newcustomtheorem}[2]{%
	\newenvironment{#1}[1]
	{%
		\renewcommand\customgenericname{#2}%
		\renewcommand\theinnercustomgeneric{##1}%
		\innercustomgeneric
	}
	{\endinnercustomgeneric}
}
\newcustomtheorem{ctheorem}{Theorem}
\newcustomtheorem{clemma}{Lemma}

\theoremstyle{definition}
\newtheorem{definition}[theorem]{Definition}

\newtheorem{example}[theorem]{Example}
\newtheorem*{example*}{Example}
\newtheorem*{examples*}{Examples}
\newtheorem{remark}[theorem]{Remark}
\newtheorem*{remark*}{Remark}
\newtheorem*{remarks*}{Remarks}
\newtheorem*{note*}{Note}

\newtheoremstyle{named}{}{}{\itshape}{}{\bfseries}{.}{.5em}{#1\thmnote{ #3}}
\theoremstyle{named}

\DeclareMathAlphabet{\mydutchcal}{U}{dutchcal}{m}{n}
\DeclareMathAlphabet{\myrsfso}{U}{rsfso}{m}{n}
\DeclareMathAlphabet{\mybbm}{U}{BOONDOX-ds}{m}{n}

\newcommand{\OEIS}[1]{\href{https://oeis.org/#1}{#1}}

\newcommand{\Fis}{\operatorname{Fis}}

\newcommand{\MatF}{\operatorname{FM}}
\newcommand{\MatP}{\operatorname{PM}}

\newcommand{\col}{\operatorname{col}}
\newcommand{\row}{\operatorname{row}}

\newcommand{\dist}{\operatorname{dist}}

\newcommand{\inv}{\operatorname{inv}}

\newcommand{\qbinom}[2]{{#1\brack #2}}
\newcommand{\stirlingb}[2]{{#1\brace #2}}

\newcommand{\sw}{v}

\newcommand{\bI}{\mathbf{I}}

\newcommand{\IPPM}{\operatorname{IPPM}}

\newcommand{\NDIPPM}{\operatorname{NDIPPM}}
\newcommand{\NDPM}{\operatorname{NDPM}}

\newcommand{\odd}{\operatorname{odd}}
\newcommand{\blk}{\operatorname{blk}}

\newcommand{\len}{\operatorname{len}}
\newcommand{\semilen}{\operatorname{semilen}}
\newcommand{\comp}{\operatorname{comp}}
\newcommand{\level}{\operatorname{level}}

\newcommand{\x}{\operatorname{x}}

\newcommand{\touch}{\operatorname{ret}}

\newcommand{\cM}{\myrsfso{M}}
\newcommand{\bbM}{\mathbb{M}}

\newcommand{\cD}{\myrsfso{D}}
\newcommand{\bbD}{\mathbb{D}}

\newcommand{\bfx}{\mathbf{x}}

\newcommand{\rmM}{\mathrm{M}}

\newcommand{\rU}{\mathrm{U}}
\newcommand{\rD}{\mathrm{D}}
\newcommand{\rL}{\mathrm{L}}

\newcommand{\lastddots}{%
	\raisebox{\dimexpr1ex-\height}{%
		$\displaystyle
		\raisebox{.5\height}{$\ddots$}
		$%
	}%
}

\title{Signed counting of partition matrices}

\author[S. Chern]{Shane Chern}
\address[S. Chern]{Fakult\"at f\"ur Mathematik, Universit\"at Wien, Oskar-Morgenstern-Platz 1, Wien 1090, Austria}
\email{chenxiaohang92@gmail.com, xiaohangc92@univie.ac.at}

\author[S. Fu]{Shishuo Fu}
\address[S. Fu]{{}\textsuperscript{(1)}College of Mathematics and Statistics \& Center for Discrete Mathematics, Chongqing University, Chongqing 401331, P.R. China \newline \indent{}\textsuperscript{(2)}Key Laboratory of Nonlinear Analysis and its Applications (Chongqing University), Ministry of Education, Chongqing 401331, P.R. China}
\email{fsshuo@cqu.edu.cn}

\date{}

\keywords{Fishburn matrix, partition matrix, inversion sequence, Stirling number of the second kind, Motzkin path, weighted Motzkin path, generating function, bijection.}

\subjclass[2020]{05A05, 05A15, 05A19}

\begin{document}
	
\sloppy

\begin{abstract}
	We prove that the signed counting (with respect to the parity of the ``$\operatorname{inv}$'' statistic) of partition matrices equals the cardinality of a subclass of inversion sequences. In the course of establishing this result, we introduce an interesting class of partition matrices called improper partition matrices. We further show that a subset of improper partition matrices is equinumerous with the set of Motzkin paths. Such an equidistribution is established both analytically and bijectively.
\end{abstract}

\maketitle

\section{Introduction}

\subsection{Background}\label{sec:background}

A \emph{Fishburn matrix} is an upper-triangular square matrix over nonnegative integers such that each row and column contains at least one nonzero entry. The \emph{weight} of such a matrix $M$, denoted by $w(M)$, is defined as the sum of all its entries. Also, the \emph{dimension} of $M$, denoted by $\dim(M)$, is the dimension of a square matrix in the usual sense (i.e., the number of rows/columns). For example, the following Fishburn matrix
\begin{align}\label{eq:F-Max-ex}
	\begin{pmatrix}
		1 & 0 & 2 & 0\\
		& 2 & 1 & 0\\
		&   & 0 & 1\\
		&   &   & 1
	\end{pmatrix}
\end{align}
has weight $8$ and dimension $4$. In this work, we denote by $\MatF_{n}$ the set of Fishburn matrices of weight $n$ and write $\MatF:= \cup_{n\ge 1} \MatF_n$.

Historically, this concept was first considered by Fishburn~\cite{Fis1983,Fis1985}, arising naturally from his earlier study of interval orders in \cite{Fis1970}. In recent years, Fishburn matrices have been connected to other combinatorial objects such as ascent sequences \cite{BCDK2010}, pattern-avoiding permutations \cite{BCDK2010,CYZ2019}, and pattern-avoiding insertion tables \cite{Lev2013}. In addition, their counting function, known as the \emph{Fishburn number}~\cite[\OEIS{A022493}]{OEIS}, plays a motivating role in topics including $q$-hypergeometric series \cite{AJ2014}, quantum modular forms \cite{Zag2010}, and Vassiliev invariants \cite{Sto1998,Zag2001}.

In 2011, Claesson, Dukes, and Kubitzke~\cite{CDK2011} introduced two matrix analogs for set partitions; the one we are interested in is the partition matrix.

Let $[n]:=\{1,2,\ldots,n\}$ throughout this paper.

\begin{definition}\label{def:pm}
	A \emph{partition matrix} $P$ on $[n]$ is an upper-triangular square matrix over the power set of $[n]$ such that
	\begin{itemize}[itemindent=*, leftmargin=*, itemsep=2pt, topsep=2pt, parsep=2pt]
		\item each row and column contains at least one nonempty subset of $[n]$;
		
		\item these nonempty subsets partition $[n]$;
		
		\item defining the \emph{column index} $\col(k)=\col_P(k)$ (resp.~\emph{row index} $\row(k)=\row_P(k)$) of $k\in [n]$ as the column (resp.~row) in $P$ that contains the subset including $k$, then for any two elements $i,j\in [n]$ with $i<j$, we always have $\col_P(i)\le \col_P(j)$.
	\end{itemize}
\end{definition}

For a partition matrix $P$ on $[n]$, we define its \emph{weight} $w(P)$ as $n$ and its \emph{dimension} by $\dim(P)$. As an example, the following matrix
\begin{align}\label{eq:P-Mat-ex}
	\begin{pmatrix}
		\{1\} & \varnothing & \{4,5\} & \varnothing\\
		& \{2,3\} & \{6\} & \varnothing\\
		& & \varnothing & \{8\}\\
		& & & \{7\}
	\end{pmatrix}
\end{align}
is a partition matrix on $\{1,2,\ldots,8\}$, and has weight $8$ and dimension $4$. Throughout, let $\MatP$ denote the set of all partition matrices and $\MatP_{n}$ the set of partition matrices on $[n]$.

Notably, partition matrices have a natural connection with Fishburn matrices. To be precise, given a partition matrix on $[n]$, if we switch its entries from the subsets of $[n]$ to the cardinalities of these subsets, then the resulting matrix is a Fishburn matrix of the same weight and dimension, just like the matrix in \eqref{eq:P-Mat-ex} reduces to that in \eqref{eq:F-Max-ex}. Such a matrix is called the \emph{induced Fishburn matrix} of our partition matrix.

On the other hand, unlike Fishburn numbers, which have no simple closed expression,\footnote{However, as shown by Zagier~\cite[p.~948, Theorem~1]{Zag2001}, the generating function for Fishburn numbers can still be simply expressed as $$\sum_{n\ge 0} \prod_{i=1}^n \big(1-(1-t)^i\big);$$ this is a typical example of quantum modular forms (after replacing $1-t$ with $t$).} the number of partition matrices of weight $n$ is the factorial of $n$. Recall that an \emph{inversion sequence} is of the form
\begin{align*}
	\text{$e := (e_1,e_2,\ldots,e_n)$ with $0\le e_i \le i-1$ for every $i$ with $1\le i\le n$};
\end{align*}
we call $n$ the \emph{length} of this inversion sequence, denoted by $\len(e)$. Let $\bI_n$ be the set of inversion sequences of length $n$. The following elegant result is due to Claesson, Dukes, and Kubitzke~\cite[p.~1626, Theorem~3]{CDK2011}.

\begin{theorem}[Claesson--Dukes--Kubitzke]\label{th:CDK}
	For every $n\ge 1$, there is a bijection $\Pi = \Pi_n$ between partition matrices of weight $n$ and inversion sequences of length $n$. As a consequence,
	\begin{align*}
		|\MatP_n| = n!.
	\end{align*}
\end{theorem}

The Claesson--Dukes--Kubitzke bijection $\Pi_n:\MatP_n\to \bI_n$ is \emph{natural} because of its simplicity in terms of construction. Precisely speaking, for a matrix $P\in \MatP_{n}$, the corresponding inversion sequence $\Pi_n(P)=(e_1,\ldots,e_n)$ is such that for each $j$ with $1\le j\le n$,
\begin{align*}
	e_j := \sum_{d=1}^{\row(j)-1} n_d,
\end{align*}
where $n_d$ denotes the total cardinality of the sets in the $d$-th column. Here the empty summation (i.e., when $\row(j)=1$) is understood to equal $0$ as usual.

\begin{example}\label{ex:CDK-bij}
	The following examples showcase how the Claesson--Dukes--Kubitzke bijection works between $\MatP_3$ and $\bI_3$:
	\begin{align*}
		\begin{pmatrix}
			\{1,2,3\}
		\end{pmatrix} & \mapsto (0,0,0), & \qquad\qquad \begin{pmatrix}
			\{1\} & \varnothing \\
			& \{2,3\}
		\end{pmatrix} & \mapsto (0,1,1),\\
		\begin{pmatrix}
			\{1\} & \{2\} \\
			& \{3\}
		\end{pmatrix} & \mapsto (0,0,1), & \begin{pmatrix}
			\{1,2\} & \varnothing \\
			& \{3\}
		\end{pmatrix} & \mapsto (0,0,2),\\
		\begin{pmatrix}
			\{1\} & \{3\} \\
			& \{2\}
		\end{pmatrix} & \mapsto (0,1,0), & \begin{pmatrix}
			\{1\} & \varnothing & \varnothing \\
			& \{2\} & \varnothing \\
			& & \{3\}
		\end{pmatrix} & \mapsto (0,1,2).
	\end{align*}
	As a concrete illustration, let us choose
	\begin{align*}
		P = \begin{pmatrix}
			\{1\} & \{3\} \\
			& \{2\}
		\end{pmatrix}\in \MatP_3.
	\end{align*}
	We have the row indices $\row(1)=1$, $\row(2)=2$, and $\row(3)=1$. And the total cardinality in the two columns is $n_1 = 1$ and $n_2 = 2$. Now for $\Pi_3(P) = (e_1, e_2, e_3) \in \bI_3$, the Claesson--Dukes--Kubitzke bijection gives us $e_1 = e_3 = 0$ (because of the empty sum) and $e_2 = n_1 = 1$.
\end{example}

\subsection{Main result and motivation}\label{sec:main-result}

To more closely tie partition matrices and Fishburn matrices, as well as the counting of them, we consider the following polynomial in $\mathbb{N}[q]$:
\begin{align}\label{eq:Sn-def}
	S_n(q) := \sum_{M\in \MatF_{n}} \prod_{j=1}^{\dim(M)} \qbinom{m_{1,j}+m_{2,j}+\cdots+m_{j,j}}{m_{1,j},m_{2,j},\ldots,m_{j,j}}_q,
\end{align}
where the sum runs over all Fishburn matrices $M:=(m_{i,j})_{i,j\ge 1}$ of weight $n$. Here the \emph{$q$-multinomial coefficients} are defined by
\begin{align*}
	\qbinom{a_1+\cdots+a_k}{a_1,\ldots,a_k}_q := \frac{[a_1+\cdots+a_k]_q!}{[a_1]_q!\cdots [a_k]_q!} \in \mathbb{N}[q],
\end{align*}
with $[0]_q!:=1$ and $[a]_q! := \prod_{i=1}^a \frac{1-q^i}{1-q}$ for any positive integer $a$. Meanwhile, given a partition matrix $P$ on $[n]$, a pair of numbers $(i,j)$ over $[n]$ is said to be an \emph{inversion} in $P$ if
\begin{itemize}[itemindent=*, leftmargin=*, itemsep=2pt, topsep=2pt, parsep=2pt]
	\item $i>j$;
	
	\item $\col(i)=\col(j)$;
	
	\item $\row(i)<\row(j)$.
\end{itemize}
We denote by $\inv(P)$ the number of inversions in $P$. According to a standard result on $q$-multinomial coefficients \cite[p.~41, Theorem~3.6]{And1998}, it is true that
\begin{align}\label{eq:Sn-def-2}
	S_n(q) = \sum_{P\in \MatP_{n}} q^{\inv(P)}.
\end{align}

Taking $q=0$ in \eqref{eq:Sn-def} and $q=1$ in \eqref{eq:Sn-def-2}, respectively, while for the former noting further that $\qbinom{a_1+\cdots+a_k}{a_1,\ldots,a_k}_0 = 1$, we have the relations:
\begin{align}
	S_n(0) &= |\MatF_{n}| = \Fis_n,\label{eq:S0}\\
	S_n(1) &= |\MatP_{n}| = n!.\label{eq:S1}
\end{align}
Here, $\Fis_n$ is the $n$-th Fishburn number. It is then natural to ask for an explicit expression of $S_n(q)$, at least at some typical values of $q$ other than zero and one.

The above discussion motivates us to consider the following \emph{signed} enumeration of partition matrices:
\begin{align}
	S_n(-1) = \sum_{P\in \MatP_{n}} (-1)^{\inv(P)}.
\end{align}
The initial values of this counting function are
\begin{align*}
	1, 2, 4, 10, 28, 88, 304, 1144, \ldots.
\end{align*}
After searching it in the On-Line Encyclopedia of Integer Sequences (OEIS, \cite{OEIS}), we find that this sequence coincides with \cite[\OEIS{A229046}]{OEIS}. This observation leads us to discover one of the main results in the present work.

\begin{theorem}\label{thm:q=-1}
	For every $n\ge 1$, the number $S_n(-1)$ equals the number of inversion sequences in $\bI_n(-,-,=)$, that is, inversion sequences $(e_1,\ldots,e_n)$ of length $n$ such that there is no triple $i < j < k$ with $e_i = e_k$.
\end{theorem}

\begin{remark}\label{rmk:q=-1}
	The numbers $|\bI_n(-,-,=)|$ give one of the combinatorial explanations of the sequence \cite[\OEIS{A229046}]{OEIS}; Martinez and Savage proposed this relation as an open problem in \cite[Section~2.13]{MS2018}. In the same work, Martinez and Savage also considered the ``$\dist$'' statistic for sequences in $\bI_n(-,-,=)$ counting \emph{distinct} elements in a given inversion sequence, and obtained a recurrence for the corresponding bi-parametric counting function. Later on, Cao, Jin, and Lin~\cite[p.~96, Proposition~7.2]{CJL2019}\footnote{In \cite{CJL2019}, the ``$\dist$'' statistic only counts distinct \emph{positive} elements, and hence it is one less than that of Martinez and Savage because $0$ always appears in an inversion sequence. Throughout this work, we adopt the convention of Martinez and Savage and therefore modify \eqref{eq:I-formula} from its original form in \cite{CJL2019} accordingly.} confirmed Martinez and Savage's hunch by establishing the following relation:
	\begin{align}\label{eq:I-formula}
		\sum_{e\in \bI_n(-,-,=)} z^{\dist(e)} = \sum_{j=1}^{n+1} (j-1)! \stirlingb{n+2-j}{j}z^{n+1-j},
	\end{align}
	where $\stirlingb{n}{m}$ are the \emph{Stirling numbers of the second kind}. Alternatively, as claimed by Heinz~\cite[A229046]{OEIS} and confirmed by Cao, Jin, and Lin~\cite[p.~96, Theorem~7.1]{CJL2019}, the number $|\bI_n(-,-,=)|$ also equals the number of set partitions of $[n+1]$ such that the absolute difference between the smallest elements of consecutive blocks is always greater than $1$.
\end{remark}

\subsection{Outline}\label{sec:outline}

We will take two steps to establish Theorem~\ref{thm:q=-1}:
\begin{enumerate}[label={\textup{(\arabic{*})}}, leftmargin=*, labelsep=5pt, align=left, itemsep=2pt, topsep=2pt, parsep=2pt]
	\item Find a subset of $\MatP_{n}$ such that its cardinality is given by $S_n(-1)$;
	
	\item Show that this subset is equinumerous with $\bI_n(-,-,=)$.
\end{enumerate}
In Section~\ref{sec:MIP}, we construct a subclass of partition matrices called \emph{improper partition matrices} in Definition~\ref{def:MIP} to fulfill our objective in Step~(1). Then in Section~\ref{sec:MIP-GF}, a bivariate generating function is established for our improper partition matrices. As shown in Theorem~\ref{th:MIP-I}, this generating function is identical to a bivariate generating function for the desired inversion sequences in $\bI_n(-,-,=)$, thereby completing Step~(2).

It is notable that improper partition matrices themselves are of independent research interest. Therefore, Section~\ref{sec:NIPM} is devoted to a further subset of improper partition matrices. This subset, surprisingly, is equinumerous with the set of Motzkin paths. In particular, such an equidistribution can be strengthened by accommodating two extra statistics on both sides; this subtle relation is presented in Theorem~\ref{th:NDIPPM-Motzkin}. Moreover, two proofs are provided, with one analytic and the other bijective. Both proofs build upon a combinatorial mechanism involving weighted Motzkin paths to be introduced in Section~\ref{sec:weighted-M}.

We close this paper with some questions for future research in Section~\ref{sec:outlook}.

\subsection{List of notation}

To facilitate the understanding of the paper, we list the necessary notation as follows. Here only typical notation is presented, and in particular, we omit subscripts indicating weight, length, etc. For example, we list $\MatF$ instead of $\MatF_n$ for the set of Fishburn matrices.

\begin{itemize}[
	itemindent=0pt, 
	leftmargin=2.75cm, 
	topsep=6pt, 
	itemsep=2pt,
	labelsep=0pt, 
	labelwidth=2.75cm,
	align=left
	]
	
	\item[\textbf{\textsf{Matrices}}] 
	
	\item[$\clubsuit$ \textit{Sets}]
	
	\item[$\MatF$] set of Fishburn matrices (Section~\ref{sec:background})
	
	\item[$\MatP$] set of partition matrices (Definition~\ref{def:pm})
	
	\item[$\IPPM$] set of improper partition matrices (Definition~\ref{def:MIP})
	
	\item[$\NDPM$] set of nondecreasing partition matrices (Definition~\ref{def:NDPM})
	
	\item[$\NDIPPM$] set of nondecreasing improper partition matrices (Section~\ref{sec:th-NDIPPM-strong})
	
	\item[$\blacklozenge$ \textit{Statistics}]
	
	\item[$w$] weight of a Fishburn or partition matrix (Section~\ref{sec:background})
	
	\item[$\dim$] dimension of a Fishburn or partition matrix (Section~\ref{sec:background})
	
	\item[$\row$, $\col$] row/column index of a number in a partition matrix (Definition~\ref{def:pm})
	
	\item[$\inv$] number of inversions in a partition matrix (Section~\ref{sec:main-result})
	
	\item[$v$] semi-weight of a partition matrix (Definition~\ref{def:semi-weight})
	
	\item[$\blk$] number of irreducible diagonal blocks in a nondecreasing partition matrix (Section~\ref{sec:CDK-block})
	
	\item[$\odd$] number of nonempty entries of an odd cardinality in a nondecreasing partition matrix (Section~\ref{sec:th-NDIPPM-strong})
	
	\item[\textbf{\textsf{Inversion sequences}}]
	
	\item[$\clubsuit$ \textit{Sets}]
	
	\item[$\bI$] set of inversion sequences (Section~\ref{sec:background})
	
	\item[$\bI(-,-,=)$] set of inversion sequences $e$ such that there is no triple $i < j < k$ with $e_i = e_k$ (Theorem~\ref{thm:q=-1})
	
	\item[$\blacklozenge$ \textit{Statistics}]
	
	\item[$\len$] length of an inversion sequence (Section~\ref{sec:background})
	
	\item[$\dist$] number of distinct elements in an inversion sequence (Remark~\ref{rmk:q=-1})
	
	\item[\textbf{\textsf{Lattice paths}}]
	
	\item[$\clubsuit$ \textit{Sets}]
	
	\item[$\bbD$] set of Dyck paths (Section~\ref{sec:CDK-block})
	
	\item[$\bbM$] set of Motzkin paths (Definition~\ref{def:M-path})
	
	\item[$\bbM^\dagger$] set of Motzkin paths with no level steps on the $x$-axis (Section~\ref{sec:NDIPPM-comb})
	
	\item[{$[\bbM]$}] set of restricted weighted Motzkin paths (Definition~\ref{def:weighted-M})
	
	\item[$\blacklozenge$ \textit{Statistics}]
	
	\item[$\semilen$] semilength of a Dyck path (Section~\ref{sec:CDK-block})
	
	\item[$\touch$] number of times a Dyck path returns to the $x$-axis (Section~\ref{sec:CDK-block})
	
	\item[$\len$] length of a Motzkin path (Section~\ref{sec:th-NDIPPM-strong})
	
	\item[$\level$] number of level steps on a Motzkin path (Section~\ref{sec:th-NDIPPM-strong})
	
	\item[$\level^{\x}$] number of level steps on a Motzkin path lying on the $x$-axis (Section~\ref{sec:th-NDIPPM-strong})
	
	\item[$\comp$] number of (down and level) steps on a Motzkin path ending on the $x$-axis (Section~\ref{sec:th-NDIPPM-strong})
	
	\item[$|\cdot|$] weight of a weighted Motzkin path (Definition~\ref{def:weighted-M-path})
	
	\item[$\odd$] number of odd weights in a weighted Motzkin path (Definition~\ref{def:weighted-M-path})
	
	\item[\textbf{\textsf{Mappings}}]
	
	\item[$\Pi$] the Claesson--Dukes--Kubitzke bijection from $\MatP$ to $\bI$ (Theorem~\ref{th:CDK})
	
	\item[$\Theta$] weight-preserving involution on $\MatP$ that fixes $\IPPM$ (Definition~\ref{def:Theta} and Theorem~\ref{thm:MIP=S(-1)})
	
	\item[$\Xi$] bijection from nondecreasing partition matrices to weighted Motzkin paths (Theorem~\ref{th:Xi})
	
	\item[$\Phi_n$] bijection from $[\bbM_{\ge 0}]_n$ to $\bbM_n$ (Theorem~\ref{th:NDIPPM-comb})
	
	\item[$\Omega_n$] bijection from $[\bbM_{\ge 0}^\dagger]_n$ to $[\bbM_{\ge 0}]_{n-2}$ (Lemma~\ref{le:key-dagger})
	
\end{itemize}

\section{Improper partition matrices}\label{sec:MIP}

As we have promised, in this section, we are going to determine a subset of partition matrices, the cardinality of which is counted by $S_n(-1)$. To begin with, we introduce further statistics on partition matrices. Then such a subset is precisely the set of fixed points under an involution $\Theta$ that we are going to construct.

Given a partition matrix $P$ on $[n]$, the element $i$ with $1\le i<n$ is called a \emph{descent} (resp.~an \emph{ascent}) in $P$ if the pair $(i,i+1)$ satisfies that
\begin{itemize}[itemindent=*, leftmargin=*, itemsep=2pt, topsep=2pt, parsep=2pt]
	\item $\col(i)=\col(i+1)$;
	
	\item $\row(i)>\row(i+1)$ (resp.~$\row(i)<\row(i+1)$).
\end{itemize}
And $(i,i+1)$ is said to form a \emph{descent pair} (resp.~an \emph{ascent pair}) in $P$.

\begin{definition}\label{def:MIP}
	Given a partition matrix $P\in \MatP$, a descent or an ascent $i$ of $P$ is said to be \emph{proper} if $i\equiv j\pmod 2$, where $j$ is the minimal element in the $\col(i)$-th column of $P$. Otherwise, it is said to be \emph{improper}. Let $\IPPM$ denote the subset of $\MatP$ consisting of all \emph{improper partition matrices},\footnote{This acronymic format of nomenclature, which is adopted throughout this work, is widely used in enumerative combinatorics. See, for example, the works of Andrews~\cite{And1994} and Stanley~\cite{Sta1986} for plane partitions, and the \textit{Annals} paper of Kuperberg~\cite{Kup2002} for alternating sign matrices.} which are partition matrices with every descent and ascent (if any) improper; we also denote by $\IPPM_{n}$ the subset of matrices in $\IPPM$ of weight $n$.
\end{definition}

For instance, in the earlier example, the matrix in \eqref{eq:P-Mat-ex} has an improper ascent at $5$ and a proper descent at $7$. As another example, among the six matrices in $\MatP_3$ in Example~\ref{ex:CDK-bij}, only the following four are improper:
\begin{align*}
	&\begin{psmallmatrix}
		\{1,2,3\}
	\end{psmallmatrix},\quad
	\begin{psmallmatrix}
		\{1,2\} & \varnothing \\
		& \{3\}
	\end{psmallmatrix},\quad
	\begin{psmallmatrix}
		\{1\} & \varnothing \\
		& \{2,3\}
	\end{psmallmatrix},\quad
	\begin{psmallmatrix}
		\{1\} & \varnothing & \varnothing \\
		& \{2\} & \varnothing \\
		& & \{3\}
	\end{psmallmatrix}.
\end{align*}

\begin{definition}\label{def:Theta}
	For every $n\ge 1$, let $\Theta$ be a mapping from $\MatP$ to itself, with the image $\Theta(P)$ constructed for any $P\in\MatP$ as follows:
	\begin{itemize}[itemindent=*, leftmargin=*, itemsep=2pt, topsep=2pt, parsep=2pt]
		\item if $P\in\IPPM$, then simply set $\Theta(P):=P$;
		\item if $P\in \MatP\setminus\IPPM$, find the smallest proper ascent or descent, say $i$, in $P$, and switch $i$ and $i+1$ to obtain $\Theta(P)$.
	\end{itemize}
\end{definition}

The following theorem indicates that $\IPPM$ is our desired subset for Step~(1) in Section~\ref{sec:outline}.

\begin{theorem}\label{thm:MIP=S(-1)}
	The mapping $\Theta: \MatP\to \MatP$ is a weight-preserving involution that fixes $\IPPM$, and for any $P\in\MatP\setminus\IPPM$, we have
	\begin{align}\label{Theta-sign}
		|\inv(P)-\inv(\Theta(P))|=1.
	\end{align}
	Moreover, for every $P\in\IPPM$,
	\begin{align}\label{MIP-pos}
		\inv(P)\equiv 0\pmod 2.
	\end{align} Consequently, we see that for every $n\ge 1$,
	\begin{align}\label{eq:S-1-MIP}
		|\IPPM_{n}|=S_n(-1).
	\end{align}
\end{theorem}

\begin{proof}
	It is trivial that for $P\in \IPPM$, the mapping $\Theta$ is a weight-preserving involution because $\Theta(P)$ is $P$ itself. Meanwhile, we note that for any $P\in\MatP\setminus\IPPM$, the smallest proper descent (resp.~ascent) in $P$ becomes the smallest proper ascent (resp.~descent) in $\Theta(P)$. It should then be clear from the construction of $\Theta$ that
	\begin{align*}
		w(P)=w(\Theta(P))
	\end{align*}
	and
	\begin{align*}
		\Theta(\Theta(P))=P.
	\end{align*}
	As such, $\Theta$ preserves weight and is an involution in both cases.
	
	Now we examine the pair $(P,\Theta(P))$ for $P\in\MatP\setminus\IPPM$. According to the definition of $\Theta$, the only difference between $P$ and $\Theta(P)$ is that there exists a unique pair $(i,i+1)$ so that their positions in the matrices are switched going from $P$ to $\Theta(P)$. We first suppose that $i$ is a proper ascent in $P$; then it becomes a proper descent in $\Theta(P)$. Moreover, for any third element $j$ with $j<i$ in the $\col_P(i)$-th column such that $\row_P(j)<\row_P(i)$ and $\row_P(j)<\row_P(i+1)$, we also have $\row_{\Theta(P)}(j) < \row_{\Theta(P)}(i)$ and $\row_{\Theta(P)}(j) < \row_{\Theta(P)}(i+1)$ because $\row_{\Theta(P)}(j) = \row_{P}(j)$ while $\row_{\Theta(P)}(i) = \row_P(i+1)$ and $\row_{\Theta(P)}(i+1) = \row_P(i)$ by our construction. Hence, neither of the pairs $(i,j)$ and $(i+1,j)$ form an inversion in $P$ and $\Theta(P)$. Other possibilities of elements $j\not\in \{i,i+1\}$ in the $\col_P(i)$-th column and possibilities of the row index $\row_P(j)$ can be individually checked in the same way. We conclude that $\inv(P)+1=\inv(\Theta(P))$ in this scenario. The case of $(i,i+1)$ being a proper descent in $P$ can be argued in a similar fashion, and the relation $\inv(P)-1=\inv(\Theta(P))$ follows. We have proven \eqref{Theta-sign}.
	
	Next, take any improper matrix $P$. If $\inv(P)=0$, then \eqref{MIP-pos} trivially holds true. Otherwise, suppose $(i,j)$ is an inversion in the $c$-th column of $P$ so that $i>j$ and $\row_P(i)<\row_P(j)$. Let $S$ be the subset that contains $j$. We further assume that 
	$$(j-\alpha,\ldots,j-1,j,j+1,\ldots,j+\beta)$$
	is the maximal sequence of consecutive integers containing $j$ inside $S$ for certain $\alpha,\beta\ge 0$. On one hand, $j-\alpha$ must have the same parity as the smallest element, say $c_{\min}$, in the $c$-th column; otherwise, $j-\alpha-1$, which shares the same parity with $c_{\min}$, also lies in the $c$-th column and becomes a proper ascent or descent, thereby resulting in a contradiction. On the other hand, $j+\beta$ is not the largest element in the $c$-th column; otherwise, we cannot find the element $i$ to create our inversion $(i,j)$. Thus, $j+\beta+1$ is also in the $c$-th column, and in particular, $\row_P(j+\beta)\neq\row_P(j+\beta+1)$, implying that $j+\beta$ is either an ascent or a descent. Now since $P$ is improper, $j+\beta$ must be improper, thereby having a different parity to $c_{\min}$, and hence to $j-\alpha$. We are then led to the fact that 
	$$|\{j-\alpha,\ldots,j,\ldots,j+\beta\}|\equiv 0\pmod 2.$$
	Consequently, each of $(i,j-\alpha),\ldots,(i,j),\ldots,(i,j+\beta)$ is an inversion pair, contributing collectively an even count to $\inv(P)$, so we obtain \eqref{MIP-pos} accordingly.
	
	Finally, we note that due to \eqref{Theta-sign}, each pair $(P,\Theta(P))$ jointly contributes zero in the signed counting $S_n(-1)$ provided that $P\in\MatP\setminus\IPPM$, while \eqref{MIP-pos} tells us that every improper matrix in $\IPPM$ contributes one in the calculation of $S_n(-1)$. These two facts imply \eqref{eq:S-1-MIP} and the proof is now completed.
\end{proof}

\begin{remark}\label{rmk:no-bij}
	It is clear from \eqref{eq:S-1-MIP} that Theorem~\ref{thm:q=-1} is equivalent to
	\begin{align}\label{eq:IPPMn-In}
		|\IPPM_{n}| = |\bI_n(-,-,=)|.
	\end{align}
	Recall that the Claesson--Dukes--Kubitzke bijection $\Pi$ maps partition matrices to inversion sequences. Therefore, a natural question is how this mapping behaves when it is restricted to improper partition matrices. Unfortunately, we have
	\begin{align*}
		\Pi(\IPPM_{n}) \ne \bI_n(-,-,=).
	\end{align*}
	As an instance, in Example~\ref{ex:CDK-bij}, we have shown
	\begin{align*}
		\Pi\begin{pmatrix}
			\{1,2,3\}
		\end{pmatrix} = (0,0,0) =: (e_1, e_2, e_3),
	\end{align*}
	while the latter is \emph{not} in $\bI_3(-,-,=)$ because $e_1=e_3$. This observation indicates that even if there exists a bijection between $\IPPM_{n}$ and $\bI_n(-,-,=)$, it cannot be as simple as the one of Claesson--Dukes--Kubitzke. In addition, by checking the objects in $\IPPM_{n}$ and $\bI_n(-,-,=)$ for some slightly larger $n$, we believe that this provisional bijection should be much more involved.
\end{remark}

\section{Proof of Theorem~\ref{thm:q=-1}}\label{sec:MIP-GF}

As we have discussed in Remark~\ref{rmk:no-bij}, a bijective treatment of \eqref{eq:IPPMn-In} is out of reach. For the moment, we turn to an analytic strategy to establish \eqref{eq:IPPMn-In}. Our proof builds upon a bivariate generating function for improper partition matrices, in which we also include the semi-weight statistic to be introduced below.

\begin{definition}\label{def:semi-weight}
	The \emph{semi-weight} of a partition matrix $P\in \MatP$ is defined as
	\begin{align*}
		\sw(P) := \sum_{d=1}^D \left\lceil \frac{n_d}{2} \right\rceil,
	\end{align*}
	where $D$ is the dimension of $P$, and each $n_d$ denotes the total cardinality of the sets in the $d$-th column.
\end{definition}

\begin{lemma}\label{le:GF-MIP-MatP}
	We have
	\begin{align}\label{eq:GF-MIP-MatP}
		\sum_{Q\in \IPPM} z^{\sw(Q)} t^{w(Q)} = \sum_{P\in \MatP} z^{w(P)} t^{2w(P)} (1+t^{-1})^{\dim(P)}.
	\end{align}
\end{lemma}

To prepare for the proof of this lemma, we define the \emph{reduction} of partition matrices $P\in \MatP$ in the following way. First, it is clear from Definition~\ref{def:pm} that the $d$-th column gives a set partition for integers in the interval $[\sum_{i=1}^{d-1} n_i+1,\sum_{i=1}^{d-1} n_i+n_d]$. Now we create a new square matrix by subtracting $\sum_{i=1}^{d-1} n_i$ from every element in the $d$-th column for every $d$ with $1\le d\le D$; by doing so, the $d$-th column simply becomes a set partition of $[n_d]$. Such a new matrix, denoted by $P^-$, is called the \emph{reduced matrix} of $P$. For example, the matrix in \eqref{eq:P-Mat-ex} is reduced to
\begin{align}\label{eq:RD-Mat-ex}
	\begin{pmatrix}
		\{1\} & \varnothing & \{1,2\} & \varnothing\\
		& \{1,2\} & \{3\} & \varnothing\\
		& & \varnothing & \{2\}\\
		& & & \{1\}
	\end{pmatrix}.
\end{align}
Clearly, there is a one-to-one correspondence between partition matrices and their reductions.

\begin{proof}[Proof of Lemma~\ref{le:GF-MIP-MatP}]
	It has been indicated in the proof of Theorem~\ref{thm:MIP=S(-1)} that given any $Q\in \IPPM \subset \MatP$, in every column of the reduced matrix $Q^-$, all even elements $2i$ must be placed in the same row as $2i-1$.
	
	Now we map every $P\in \MatP$ to $2^{\dim(P)}$ matrices in $\IPPM$ as follows. First, in every column of the reduction $P^-$, we replace each number $k$ by two numbers $2k-1$ and $2k$. For example, the reduced matrix in \eqref{eq:RD-Mat-ex} becomes
	\begin{align*}
		\begin{pmatrix}
			\{1,2\} & \varnothing & \{1,2,3,4\} & \varnothing\\
			& \{1,2,3,4\} & \{5,6\} & \varnothing\\
			& & \varnothing & \{3,4\}\\
			& & & \{1,2\}
		\end{pmatrix}
	\end{align*}
	after this procedure. It is notable that in the new matrix, there are an even number of elements in each column. Next, in each column, we may either preserve the largest element or remove it; in total, there are $2^{\dim(P)}$ ways to do so. More importantly, every matrix derived in this way is the reduction of a matrix in $\IPPM$. Also, the reduction of every matrix in $\IPPM$ can be uniquely created as above.
	
	After recovering the corresponding $2^{\dim(P)}$ matrices in $\IPPM$, it is clear that their semi-weights are all equal to the weight of $P$. Therefore,
	\begin{align*}
		\sum_{Q\in \IPPM} z^{\sw(Q)} t^{w(Q)} = \sum_{P\in \MatP} z^{w(P)} t^{2w(P)} \sum_{i=0}^{\dim(P)} \binom{\dim(P)}{i} t^{-i},
	\end{align*}
	which implies the desired relation.
\end{proof}

It has been shown by Claesson, Dukes, and Kubitzke~\cite[p.~1627, Corollary~5]{CDK2011} that the ``$\dim$'' statistic on $\MatP_{n}$ is Eulerian.\footnote{Note that $\dim(P)$ for $P\in \MatP_{n}$ is supported on $[1,n]$ instead of $[0,n-1]$. This explains the extra factor $x$ on the right-hand side of \eqref{eq:dim-Eulerian}.} In other words,
\begin{align}\label{eq:dim-Eulerian}
	\sum_{P\in \MatP} x^{\dim(P)} t^{w(P)} = \sum_{n\ge 1} x E_n(x) t^n,
\end{align}
where $E_n(x)$ is the $n$-th \emph{Eulerian polynomial}. In addition, according to an identity of Frobenius~\cite{Fro1910} (see also \cite[p.~1, eq.~(3)]{FS1970}), we have
\begin{align}\label{eq:Frob}
	E_n(x) = \sum_{j=1}^n j! \stirlingb{n}{j} (x-1)^{n-j}.
\end{align}
Therefore,
\begin{align}\label{eq:MatP-gf-dim}
	\sum_{P\in \MatP} x^{\dim(P)} t^{w(P)} = \sum_{n\ge 1} \sum_{j=1}^n j! \stirlingb{n}{j} x (x-1)^{n-j} t^n.
\end{align}

Toward a proof of Theorem~\ref{thm:q=-1}, we need the following equidistribution between $\IPPM_{n}$ and $\bI_n(-,-,=)$.

\begin{theorem}\label{th:MIP-I}
	For every $n\ge 1$,
	\begin{align}\label{eq:MIP-I}
		\sum_{Q\in \IPPM_{n}} z^{\sw(Q)} = \sum_{e\in \bI_n(-,-,=)} z^{\dist(e)}.
	\end{align}
\end{theorem}

\begin{proof}
	It follows from \eqref{eq:GF-MIP-MatP} and \eqref{eq:MatP-gf-dim} that
	\begin{align}\label{eq:IPPM-gf-zt}
		\sum_{Q\in \IPPM} z^{\sw(Q)} t^{w(Q)} = \sum_{n\ge 1} \sum_{j=1}^n j! \stirlingb{n}{j} z^{n} (1+t^{-1}) t^{n+j}.
	\end{align}
	Hence, by isolating the coefficient of $t^N$ on the right-hand side of \eqref{eq:IPPM-gf-zt}, we have
	\begin{align*}
		\sum_{Q\in \IPPM_{N}} z^{\sw(Q)} = \sum_{j=1}^{N-1} j! \stirlingb{N-j}{j} z^{N-j} + \sum_{j=1}^{N} j! \stirlingb{N+1-j}{j} z^{N+1-j}.
	\end{align*}
	In view of \eqref{eq:I-formula}, it suffices to show that for every $N\ge 1$,
	\begin{align*}
		&\sum_{j=1}^{N+1} (j-1)! \stirlingb{N+2-j}{j} z^{N+1-j}\\
		&\qquad=\sum_{j=1}^{N-1} j! \stirlingb{N-j}{j} z^{N-j} + \sum_{j=1}^{N} j! \stirlingb{N+1-j}{j} z^{N+1-j}.
	\end{align*}
	Recall the following recurrence for the Stirling numbers of the second kind~\cite[p.~625, eq.~(26.8.22)]{Bre2010}:
	\begin{align*}
		\stirlingb{n}{m} = m \stirlingb{n-1}{m} + \stirlingb{n-1}{m-1}.
	\end{align*}
	Thus,
	\begin{align*}
		&\sum_{j=1}^{N+1} (j-1)! \stirlingb{N+2-j}{j} z^{N+1-j}\\
		&\quad = N! \stirlingb{1}{N+1} + \sum_{j=1}^{N} (j-1)! \left(j\stirlingb{N+1-j}{j} + \stirlingb{N+1-j}{j-1}\right)z^{N+1-j}\\
		&\quad = \sum_{j=1}^{N} j! \stirlingb{N+1-j}{j} z^{N+1-j} + \sum_{j=1}^{N} (j-1)! \stirlingb{N+1-j}{j-1} z^{N+1-j}\\
		&\quad = \sum_{j=1}^{N} j! \stirlingb{N+1-j}{j} z^{N+1-j} + \sum_{j=1}^{N-1} j! \stirlingb{N-j}{j} z^{N-j},
	\end{align*}
	where we have used the vanishing of $\stirlingb{n}{m}$ when $m=0$ or $m>n$. Our proof is therefore completed.
\end{proof}

Now we need one last push for Theorem~\ref{thm:q=-1}.

\begin{proof}[Proof of Theorem~\ref{thm:q=-1}]
	In \eqref{eq:MIP-I}, we take $z=1$ to get $|\IPPM_{n}| = |\bI_n(-,-,=)|$. Then Theorem~\ref{thm:q=-1} is a direct consequence of \eqref{eq:S-1-MIP}.
\end{proof}

\section{Nondecreasing improper partition matrices}\label{sec:NIPM}

This section is devoted to an interesting subset of improper partition matrices. To begin with, we define partition matrices that are nondecreasing, the concept of which was introduced by Claesson, Dukes, and Kubitzke in \cite[p.~1628, Section~2.2]{CDK2011}.

\begin{definition}\label{def:NDPM}
	Given a partition matrix $P$ on $[n]$, we say it is \emph{nondecreasing} if for any two elements $i,j\in [n]$ with $i<j$, we always have $\row_P(i)\le \row_P(j)$ and $\col_P(i)\le \col_P(j)$.
\end{definition}

Now we count the number of nondecreasing improper partition matrices of a fixed weight, and the enumerations read:
\begin{align*}
	1, 2, 4, 9, 21, 51, 127, 323, \ldots
\end{align*}
It is notable that these values appear in the sequence of \emph{Motzkin numbers}, given in the OEIS entry \cite[\OEIS{A001006}]{OEIS}. This observation leads us to a simplified version of the main result in the present section.

\begin{theorem}\label{th:NDIPPM-Motzkin-simple}
	For every $n\ge 1$, the number of nondecreasing improper partition matrices of weight $n$ equals the $n$-th Motzkin number.
\end{theorem}

\subsection{The Claesson--Dukes--Kubitzke bijection revisited}\label{sec:CDK-block}

In \cite[p.~1628, Theorem~8]{CDK2011}, the mapping $\Pi$ of Claesson--Dukes--Kubitzke is applied to nondecreasing partition matrices, resulting in a bijective correspondence to nondecreasing inversion sequences.

Recall that a \emph{Dyck path} $\cD$ of \emph{semilength} $n$ is a lattice path starting from $(0,0)$, ending at $(2n,0)$, and never falling below the $x$-axis such that only \emph{up} ($\nearrow$) steps $(1,1)$ and \emph{down} ($\searrow$) steps $(1,-1)$ are used; we write $\semilen(\cD) := n$. Such paths are enumerated by the $n$-th \emph{Catalan number} $\operatorname{Cat}_n := \frac{1}{n+1} \binom{2n}{n}$, listed in the OEIS entry \cite[\OEIS{A000108}]{OEIS}. Let $\bbD_n$ denote the set of Dyck paths of semilength $n$ and write $\bbD := \cup_{n\ge 1} \bbD_n$.

A bijection between nondecreasing inversion sequences of length $n$ and Dyck paths of semilength $n$ is further provided in \cite[p.~1628, Proposition~9]{CDK2011}. To be specific, given a Dyck path $\cD\in \bbD_n$, the corresponding nondecreasing inversion sequence $e=(e_1,\ldots,e_n)$ is such that $e_i$ equals the number of down steps before the $i$-th up step for every $i\in [n]$.

Let $\NDPM_n$ denote the set of nondecreasing partition matrices of weight $n$, and write $\NDPM := \cup_{n\ge 1} \NDPM_n$. In Table~\ref{tab:CDK-bij-Dyck}, we list the five matrices in $\NDPM_3$ and their corresponding Dyck paths.

\begin{table}[ht!]
	\def\arraystretch{2}
	\centering
	\caption{The correspondence between $\NDPM_3$ and $\bbD_3$}\label{tab:CDK-bij-Dyck}
	\begin{tabular}{cp{.5in}cp{.5in}c}
		\hline
		partition matrix & & inversion sequence & & Dyck path\\
		\hline
		$\begin{psmallmatrix}
			\{1,2,3\}
		\end{psmallmatrix}$ & & $(0,0,0)$ & & $\vcenter{\hbox{\begin{tikzpicture}[line width=0.8pt,scale=0.3]
					\draw[thick] (0,0)--++(1,1)--++(1,1)--++(1,1)--++(1,-1)--++(1,-1)--++(1,-1);
					\fill (0,0) circle(1ex) ++(1,1)circle(1ex) ++(1,1)circle(1ex) ++(1,1)circle(1ex)
					++(1,-1)circle(1ex) ++(1,-1)circle(1ex) ++(1,-1)circle(1ex);
		\end{tikzpicture}}}$\\
		$\begin{psmallmatrix}
			\{1\} & \{2\} \\
			& \{3\}
		\end{psmallmatrix}$ & & $(0,0,1)$ & & $\vcenter{\hbox{\begin{tikzpicture}[line width=0.8pt,scale=0.3]
					\draw[thick] (0,0)--++(1,1)--++(1,1)--++(1,-1)--++(1,1)--++(1,-1)--++(1,-1);
					\fill (0,0) circle(1ex) ++(1,1)circle(1ex) ++(1,1)circle(1ex) ++(1,-1)circle(1ex)
					++(1,1)circle(1ex) ++(1,-1)circle(1ex) ++(1,-1)circle(1ex);
		\end{tikzpicture}}}$\\
		$\begin{psmallmatrix}
			\{1,2\} & \varnothing \\
			& \{3\}
		\end{psmallmatrix}$ & & $(0,0,2)$ & & $\vcenter{\hbox{\begin{tikzpicture}[line width=0.8pt,scale=0.3]
					\draw[thick] (0,0)--++(1,1)--++(1,1)--++(1,-1)--++(1,-1)--++(1,1)--++(1,-1);
					\fill (0,0) circle(1ex) ++(1,1)circle(1ex) ++(1,1)circle(1ex) ++(1,-1)circle(1ex)
					++(1,-1)circle(1ex) ++(1,1)circle(1ex) ++(1,-1)circle(1ex);
		\end{tikzpicture}}}$\\
		$\begin{psmallmatrix}
			\{1\} & \varnothing \\
			& \{2,3\}
		\end{psmallmatrix}$ & & $(0,1,1)$ & & $\vcenter{\hbox{\begin{tikzpicture}[line width=0.8pt,scale=0.3]
					\draw[thick] (0,0)--++(1,1)--++(1,-1)--++(1,1)--++(1,1)--++(1,-1)--++(1,-1);
					\fill (0,0) circle(1ex) ++(1,1)circle(1ex) ++(1,-1)circle(1ex) ++(1,1)circle(1ex)
					++(1,1)circle(1ex) ++(1,-1)circle(1ex) ++(1,-1)circle(1ex);
		\end{tikzpicture}}}$\\
		$\begin{psmallmatrix}
			\{1\} & \varnothing & \varnothing \\
			& \{2\} & \varnothing \\
			& & \{3\}
		\end{psmallmatrix}$ & & $(0,1,2)$ & & $\vcenter{\hbox{\begin{tikzpicture}[line width=0.8pt,scale=0.3]
					\draw[thick] (0,0)--++(1,1)--++(1,-1)--++(1,1)--++(1,-1)--++(1,1)--++(1,-1);
					\fill (0,0) circle(1ex) ++(1,1)circle(1ex) ++(1,-1)circle(1ex) ++(1,1)circle(1ex)
					++(1,-1)circle(1ex) ++(1,1)circle(1ex) ++(1,-1)circle(1ex);
		\end{tikzpicture}}}$\\
		\hline
	\end{tabular}
\end{table}

As such, the following enumerative relation is offered in \cite[p.~1628, Proposition~9]{CDK2011}.

\begin{theorem}[Claesson--Dukes--Kubitzke]\label{th:NDPM-Dyck-simple}
	For every $n\ge 1$, the number of nondecreasing partition matrices of weight $n$ equals the $n$-th Catalan number.
\end{theorem}

\begin{remark}\label{rmk:CDK-NDIPPM}
	It is a classical result \cite[p.~296, Item~(M6)]{DS1977} that the $n$-th Motzkin number counts Dyck paths of semilength $n$ in which no three consecutive up steps are allowed. However, as we have seen in Table~\ref{tab:CDK-bij-Dyck}, the nondecreasing \emph{improper} partition matrix $\begin{psmallmatrix}
		\{1,2,3\}
	\end{psmallmatrix}$ is mapped to a Dyck path \emph{with} three consecutive up steps under the above bijective construction. Another standard way \cite[\OEIS{A001006}]{OEIS} to interpret Motzkin numbers with restricted Dyck paths is by requiring that all valleys have even $x$-coordinate. However, with the Claesson--Dukes--Kubitzke bijection, we have
	\begin{align*}
		\begin{psmallmatrix}
			\{1\} & \{2\} & \varnothing\\
			& \varnothing & \{3,4\}\\
			& & \{5\}
		\end{psmallmatrix} \quad\xmapsto{\Pi}\quad
		(0,0,1,1,2)
		\quad\mapsto\quad \vcenter{\hbox{\begin{tikzpicture}[line width=0.8pt,scale=0.3]
					\coordinate (O) at (0,0);
					\draw[thick] (O)--++(1,1)--++(1,1)--++(1,-1)--++(1,1)--++(1,1)--++(1,-1)--++(1,1)--++(1,-1)--++(1,-1)--++(1,-1);
					\filldraw (O) circle(0.5ex) ++(1,1)circle(0.5ex) ++(1,1)circle(0.5ex) ++(1,-1)circle(0.5ex)
					++(1,1)circle(0.5ex) ++(1,1)circle(0.5ex) ++(1,-1)circle(0.5ex) ++(1,1)circle(0.5ex) ++(1,-1)circle(0.5ex) ++(1,-1)circle(0.5ex) ++(1,-1)circle(0.5ex);
		\end{tikzpicture}}}\,\,;
	\end{align*}
	here the nondecreasing partition matrix is \emph{improper}, while the first valley of the Dyck path has \emph{odd} $x$-coordinate $3$ (recall that the Dyck path starts with $x$-coordinate $0$). Due to the difficulty of characterizing improperness, it is less likely that the Claesson--Dukes--Kubitzke bijection provides a correct combinatorial mechanism toward Theorem~\ref{th:NDIPPM-Motzkin-simple}.
\end{remark}

Now we note that a closer examination of the above bijective correspondence reveals more information about the equinumerous property in Theorem~\ref{th:NDPM-Dyck-simple}.

To begin with, we define $\blk(A)$ as the number of \emph{irreducible diagonal blocks} in a given $A\in \NDPM$. For example, the five nondecreasing matrices in Table~\ref{tab:CDK-bij-Dyck} consist of $1$, $1$, $2$, $2$, $3$ irreducible diagonal block(s), respectively.

It turns out that the Claesson--Dukes--Kubitzke bijection $\Pi$ preserves $\blk(A)$ as the number of elements $e_i$ with $e_i=i-1$ in $\Pi(A) = (e_1,\ldots,e_n)$. This statistic is further preserved as the number of times the resulting Dyck path $\cD$ \emph{returns} to the $x$-axis, denoted by $\touch(\cD)$. Therefore, the following strengthening of Theorem~\ref{th:NDPM-Dyck-simple} is apparent.

\begin{theorem}[Strengthening of Theorem~\ref{th:NDPM-Dyck-simple}]\label{th:NDPM-Dyck}
	For every $n\ge 1$, the statistic $\blk$ on $\NDPM_n$ is equidistributed with the statistic $\touch$ on $\bbD_n$.
\end{theorem}

\subsection{Strengthening of Theorem~\ref{th:NDIPPM-Motzkin-simple}}\label{sec:th-NDIPPM-strong}

Let $\NDIPPM_n$ denote the set of nondecreasing improper partition matrices of weight $n$, and write $\NDIPPM := \cup_{n\ge 1} \NDIPPM_n$. The consideration of the ``$\blk$'' statistic for nondecreasing partition matrices casts light on the possibility of elaborating on Theorem~\ref{th:NDIPPM-Motzkin-simple}. In addition, we denote by $\odd(A)$ the number of nonempty entries of an \emph{odd} cardinality in a nondecreasing partition matrix $A$. Since $\NDIPPM \subset \NDPM$, both statistics $\blk(B)$ and $\odd(B)$ are also defined for every $B\in \NDIPPM$.

Instead of the restricted Dyck paths as in Remark~\ref{rmk:CDK-NDIPPM}, a better description of the Motzkin numbers comes from Motzkin paths.

\begin{definition}\label{def:M-path}
	A \emph{Motzkin path} $\cM$ is a lattice path starting from $(0,0)$, ending on the right half of the $x$-axis (i.e., a certain $(n,0)$ with $n\ge 1$), and never falling below the $x$-axis such that only \emph{up} ($\nearrow$) steps $(1,1)$, \emph{down} ($\searrow$) steps $(1,-1)$, and \emph{level} ($\rightarrow$) steps $(1,0)$ are used.
\end{definition}

Let $\bbM$ denote the set of Motzkin paths. Furthermore, the $x$-coordinate of the ending point (i.e., the number $n$) in the above definition is called the \emph{length} of this Motzkin path $\cM$, denoted by $\len(\cM)$. Accordingly, we use $\bbM_n$ to denote the set of Motzkin paths of length $n$. On certain occasions, we also need the empty path $\varnothing$ for our analysis, while we assume $\len(\varnothing) := 0$. As such, we denote $\bbM_0 := \{\varnothing\}$, and further write $\bbM_{\ge 0} := \bbM \cup \bbM_0$.

For the moment, let us use $\#_{\cM}(\nearrow)$, $\#_{\cM}(\searrow)$, and $\#_{\cM}(\rightarrow)$ to count the number of up, down, and level steps on $\cM \in \bbM_{\ge 0}$, respectively. In addition, we write
\begin{align*}
	\level(\cM) := \#_{\cM}(\rightarrow)
\end{align*}
as a statistic. It is clear that
\begin{align*}
	\len(\cM) = \#_{\cM}(\nearrow) + \#_{\cM}(\searrow) + \#_{\cM}(\rightarrow).
\end{align*}
Finally, we denote by $\comp(\cM)$ the number of steps on $\cM$ that end on the $x$-axis. Note that a level step lying on the $x$-axis is also viewed as ``ending on the $x$-axis.'' Also, we denote by $\#_{\cM}^{\x}(\searrow)$ and $\#_{\cM}^{\x}(\rightarrow)$ the number of down and level steps ending on the $x$-axis, respectively. And as before, we introduce the statistic
\begin{align*}
	\level^{\x}(\cM) := \#_{\cM}^{\x}(\rightarrow).
\end{align*}
It is again obvious from its definition that
\begin{align*}
	\comp(\cM) = \#_{\cM}^{\x}(\searrow) + \#_{\cM}^{\x}(\rightarrow).
\end{align*}

Theorem~\ref{th:NDIPPM-Motzkin-simple} can be refined as follows.

\begin{theorem}[Strengthening of Theorem~\ref{th:NDIPPM-Motzkin-simple}]\label{th:NDIPPM-Motzkin}
	For every $n\ge 1$, the pair of statistics $(\blk,\odd)$ on $\NDIPPM_n$ is equidistributed with the pair of statistics $(\comp,\level)$ on $\bbM_n$.
\end{theorem}

To prove this equidistribution, we introduce weighted Motzkin paths in Section~\ref{sec:weighted-M}. Then both analytic and combinatorial treatments will be provided, as presented in Sections~\ref{sec:NDIPPM-anal} and \ref{sec:NDIPPM-comb}, respectively.

\subsection{Weighted Motzkin paths}\label{sec:weighted-M}

We have pointed out in Remark~\ref{rmk:CDK-NDIPPM} that the Claesson--Dukes--Kubitzke bijection may not be a suitable tool for nondecreasing improper partition matrices. As such, we turn to a different combinatorial strategy conceptualized in terms of weighted Motzkin paths.

We start by assigning a \emph{total order} to indices in $[D]^2$ for a certain $D\ge 1$:
\begin{align}\label{eq:index-total-order}
	\begin{array}{cccccccc}
		& (1,1) & < & (1,2) & < & \cdots & < & (1,D)\\
		< & (2,1) & < & (2,2) & < & \cdots & < & (2,D)\\
		< & \cdots\\
		< & (D,1) & < & (D,2) & < & \cdots & < & (D,D)
	\end{array}.
\end{align}
Now let us make some useful observations on nondecreasing partition matrices from Definitions~\ref{def:pm} and \ref{def:NDPM}.

\begin{fact}\label{fact:NDPM}
	For every nondecreasing partition matrix of weight $n$ and dimension $D$, the following statements are true:
	\begin{enumerate}[label={\textup{(\arabic{*})}}, leftmargin=*, labelsep=5pt, align=left, itemsep=2pt, topsep=2pt, parsep=2pt]
		\item The entries at position $(1,1)$ and $(D,D)$ are nonempty;
		
		\item Letting $(r,c)$ be the row/column index of a nonempty entry, then the next nonempty entry (according to the total order \eqref{eq:index-total-order}), if existing, must lie in one of the three positions: $(r+1,c)$, $(r,c+1)$, or $(r+1,c+1)$;
		
		\item If two consecutive numbers $i,i+1\in[n]$ belong to different nonempty entries with row/column indices $(r_1,c_1)$ and $(r_2,c_2)$, respectively, then $(r_1,c_1) < (r_2,c_2)$, and in addition, there is no nonempty entry such that its row/column index lies between $(r_1,c_1)$ and $(r_2,c_2)$.
	\end{enumerate}
\end{fact}

Parts~(1) and (2) in Fact~\ref{fact:NDPM} immediately imply a correspondence between nondecreasing partition matrices and certain lattice paths on a lattice such that the vector $(1,0)$ goes toward \emph{south} and the vector $(0,1)$ goes toward \emph{east};\footnote{Note that this lattice is different from the usual lattice in which the vector $(1,0)$ goes toward east and the vector $(0,1)$ goes toward north.} such a lattice may easily mimic the grid of indices of a matrix, as illustrated in \eqref{eq:induced-lattice} below:
\begin{align}\label{eq:induced-lattice}
	\vcenter{\hbox{\begin{tikzpicture}[scale=0.5]
				\draw (1,1) grid (4,4);
				\draw[->] (1,4) -- (4.5,4);
				\draw[->] (1,4) -- (1,0.5);
				\node[above left] at (1,4) {\tiny $O$};
				\node[above] at (4.5,4) {\tiny $y$};
				\node[left] at (1,0.5) {\tiny $x$};
				\draw[line width=1.5pt,draw=cyan,->] (1,4) -- (2,4);
				\draw[line width=1.5pt,draw=cyan,->] (1,4) -- (1,3);
				\node[above] at (2,4) {\tiny $(0,1)$};
				\node[left] at (1,3) {\tiny $(1,0)$};
	\end{tikzpicture}}}\,\,.
\end{align}
Now we construct a path in such a way that for every index in $[D]^2$ at which the entry is nonempty, we add a step from this index to the next index (according to the total order \eqref{eq:index-total-order}) featuring a nonempty entry. Then we arrive at a path starting at $(1,1)$, ending at $(D,D)$, and restricted within the region $\{(x,y)\in \mathbb{Z}^2: y \ge x \ge 1\}$ such that the \emph{south} ($\downarrow$) steps $(1,0)$, \emph{east} ($\rightarrow$) steps $(0,1)$, and \emph{southeast} ($\searrow$) steps $(1,1)$ are used.

It is notable that this process is many-to-one, as evidenced by the following example:
\begin{align*}
	\begin{psmallmatrix}
		\{1\} & \{2,3\} & \{4,5\} & \varnothing\\
		& \varnothing & \{6\} & \varnothing\\
		& & \varnothing & \{7,8\}\\
		& & & \{9,10\}
	\end{psmallmatrix} \quad\text{and}\quad \begin{psmallmatrix}
		\{1\} & \{2,3\} & \{4,5\} & \varnothing\\
		& \varnothing & \{6,7\} & \varnothing\\
		& & \varnothing & \{8\}\\
		& & & \{9,10\}
	\end{psmallmatrix} \quad\mapsto\quad
	\vcenter{\hbox{\begin{tikzpicture}[scale=0.5]
				\draw (1,1) grid (4,4);
				\foreach\i\j in {1/4, 2/4, 3/4, 3/3, 4/2, 4/1}
				{\filldraw[blue] (\i,\j) circle(0.75ex);}
				\draw[line width=2pt,draw=blue] (1,4) -- (2,4) -- (3,4) -- (3,3) -- (4,2) -- (4,1);
				\draw[draw=gray,dashed] (1,4)--(4,1);
	\end{tikzpicture}}}\,\,.
\end{align*}
However, we recall that the induced Fishburn matrix of the partition matrix in question records the cardinality of each \emph{nonempty} entry. Therefore, for each lattice point on the path, we may assign a \emph{positive} weight to keep track of the cardinality of the corresponding nonempty entry. By doing so, we arrive at a bijective correspondence in view of Part~(3) in Fact~\ref{fact:NDPM}. As an illustration, the previous example becomes
\begin{align}
	\begin{psmallmatrix}
		\{1\} & \{2,3\} & \{4,5\} & \varnothing\\
		& \varnothing & {\color{red}\{6\}} & \varnothing\\
		& & \varnothing & {\color{red}\{7,8\}}\\
		& & & \{9,10\}
	\end{psmallmatrix} &\quad\mapsto\quad
	\vcenter{\hbox{\begin{tikzpicture}[scale=0.5]
				\draw (1,1) grid (4,4);
				\foreach\i\j in {1/4, 2/4, 3/4, 3/3, 4/2, 4/1}
				{\filldraw[blue] (\i,\j) circle(0.75ex);}
				\draw[line width=2pt,draw=blue] (1,4) -- (2,4) -- (3,4) -- (3,3) -- (4,2) -- (4,1);
				\draw[draw=gray,dashed] (1,4)--(4,1);
				\node[above right] at (1,4) {\footnotesize $1$};
				\node[above right] at (2,4) {\footnotesize $2$};
				\node[above right] at (3,4) {\footnotesize $2$};
				\node[above right] at (3,3) {\footnotesize ${\color{red}1}$};
				\node[above right] at (4,2) {\footnotesize ${\color{red}2}$};
				\node[above right] at (4,1) {\footnotesize $2$};
	\end{tikzpicture}}}\,\,,\label{eq:ndpm-ex1}\\
	\begin{psmallmatrix}
		\{1\} & \{2,3\} & \{4,5\} & \varnothing\\
		& \varnothing & {\color{red}\{6,7\}} & \varnothing\\
		& & \varnothing & {\color{red}\{8\}}\\
		& & & \{9,10\}
	\end{psmallmatrix} &\quad\mapsto\quad
	\vcenter{\hbox{\begin{tikzpicture}[scale=0.5]
				\draw (1,1) grid (4,4);
				\foreach\i\j in {1/4, 2/4, 3/4, 3/3, 4/2, 4/1}
				{\filldraw[blue] (\i,\j) circle(0.75ex);}
				\draw[line width=2pt,draw=blue] (1,4) -- (2,4) -- (3,4) -- (3,3) -- (4,2) -- (4,1);
				\draw[draw=gray,dashed] (1,4)--(4,1);
				\node[above right] at (1,4) {\footnotesize $1$};
				\node[above right] at (2,4) {\footnotesize $2$};
				\node[above right] at (3,4) {\footnotesize $2$};
				\node[above right] at (3,3) {\footnotesize ${\color{red}2}$};
				\node[above right] at (4,2) {\footnotesize ${\color{red}1}$};
				\node[above right] at (4,1) {\footnotesize $2$};
	\end{tikzpicture}}}\,\,.\label{eq:ndpm-ex2}
\end{align}

\begin{remark}
	The lattice path in the above construction is isomorphic to a Motzkin path whose length equals one less than the number of nonempty entries in the nondecreasing partition matrix. This can be seen by a counterclockwise rotation of $45$ degrees.
\end{remark}

The above discussion motivates us to consider weighted Motzkin paths. As usual, for $\mathbf{s} = (s_1,s_2,\ldots,s_l)$ a finite number sequence, we denote its \emph{length} by $\ell(\mathbf{s}) := l$ and its \emph{weight} by $|\mathbf{s}|:=s_1+s_2+\cdots+s_l$.

\begin{definition}\label{def:weighted-M-path}
	A \emph{weighted Motzkin path}, denoted by $\cM(\bfx)$, is a Motzkin path $\cM\in \bbM_{\ge 0}$ associated with a weight sequence $\bfx$ of \emph{positive} integers such that $\ell(\bfx) = \len(\cM)+1$, while those $\len(\cM)+1$ positive integers represent the weights of the $\len(\cM)+1$ lattice points on $\cM$, respectively. If $\len(\cM) = l$, we index the weight sequence by $\bfx := (x_0,x_1,\ldots,x_l)$, and for each $i$, we assign $x_i$ to be the weight of the unique lattice point on $\cM$ with $x$-coordinate $i$.
\end{definition}

If a statistic is meaningful for a Motzkin path $\cM$, we transplant it to the relevant weighted Motzkin path $\cM(\bfx)$. For instance, $\len(\cM(\bfx)) := \len(\cM)$ and $\level(\cM(\bfx)) := \level(\cM)$. In addition, we denote the \emph{weight} of $\cM(\bfx)$ as $|\cM(\bfx)| := |\bfx|$, and the statistic $\odd(\cM(\bfx))$ as the number of odd weights in $\bfx$.

Piecing together the arguments gathered so far, we have the following statement.

\begin{theorem}\label{th:Xi}
	There is a bijection $\Xi$ from nondecreasing partition matrices to weighted Motzkin paths such that
	\begin{align}
		\{w,\blk,\odd\}\big(A\big) = \{|\cdot|,\level^{\x}+1,\odd\}\big(\Xi(A)\big)
	\end{align}
	for every nondecreasing partition matrix $A$.
\end{theorem}

\begin{proof}
	The bijective correspondence we have described earlier, which maps a nondecreasing partition matrix $A$ to a weighted Motzkin path $\cM(\bfx)$, is exactly the claimed $\Xi$. It is clear that $w(A) = |\cM(\bfx)|$ and $\odd(A) = \odd(\cM(\bfx))$. In addition, excluding the irreducible diagonal block at the left top corner, each new irreducible diagonal block in $A$ produces a new level step on the $x$-axis on $\cM$, and hence $\blk(A) = \level^{\x}(\cM(\bfx)) + 1$.
\end{proof}

\begin{remark}\label{rmk:Motz-word}
	Recall that each Motzkin path can be uniquely represented as a \emph{word} in the following way. First, for each of the three types of steps on a Motzkin path, we assign a \emph{letter}: \textbf{u}p steps $\rU:=(1,1)$, \textbf{d}own steps $\rD:=(1,-1)$, and \textbf{l}evel steps $\rL:=(1,0)$. Then each Motzkin path of length $l$ can be translated into a word of length $l$ consisting of the letters $\rU$, $\rD$, and $\rL$ by scanning the steps from left to right; as a concrete example, the Motzkin path
	\begin{align*}
		\vcenter{\hbox{\begin{tikzpicture}[line width=0.8pt,scale=0.7]
					\coordinate (O) at (0,0);
					\draw[thick] (O)--++(1,1)--++(1,1)--++(1,-1)--++(1,0)--++(1,-1)--++(1,0)--++(1,1)--++(1,-1);
					\filldraw (O) circle(0.5ex) ++(1,1)circle(0.5ex) ++(1,1)circle(0.5ex) ++(1,-1)circle(0.5ex)
					++(1,0)circle(0.5ex) ++(1,-1)circle(0.5ex) ++(1,0)circle(0.5ex) ++(1,1)circle(0.5ex) ++(1,-1)circle(0.5ex);
					\path (0.4,0.7) node {$\rU$} ++(1,1) node {$\rU$} ++(1.3,0) node {$\rD$} ++(0.8,-0.4) node {$\rL$} ++(1.2,-0.6) node {$\rD$} ++(0.8,-0.4) node {$\rL$} ++(0.8,0.4) node {$\rU$} ++(1.3,0) node {$\rD$};
		\end{tikzpicture}}}
	\end{align*}
	corresponds to the word $(\rU, \rU, \rD, \rL, \rD, \rL, \rU, \rD)$. Using this notation, we write a Motzkin path $\cM$ of length $l$ as
	\begin{align*}
		\cM := (\rmM_1,\rmM_2,\ldots,\rmM_l)
	\end{align*}
	with each $\rmM_i \in \{\rU,\rD,\rL\}$, and it is further possible to illustrate a weighted Motzkin path
	\begin{align*}
		\cM(\bfx) := (\rmM_1,\ldots,\rmM_l)(x_0,x_1,\ldots,x_l)
	\end{align*}
	as
	\begin{align}\label{eq:weighted-Motz-word}
		\begin{array}{ccccccccccc}
			& & & 1 & & \cdots & & \cdots & & l\\
			\cM & & & \rmM_1 & & \cdots & & \cdots & & \rmM_l\\
			\bfx & & x_0 & & x_1 & & \cdots & & x_{l-1} & & x_l
		\end{array}
	\end{align}
	with the first row, which may be omitted, labeling the indices of the steps $\rmM_i$.
\end{remark}

\begin{example}
	In terms of the mapping $\Xi$, the correspondence in \eqref{eq:ndpm-ex1} should be more formally expressed as the following weighted Motzkin path:
	\begin{align*}
		\begin{psmallmatrix}
			\{1\} & \{2,3\} & \{4,5\} & \varnothing\\
			& \varnothing & \{6\} & \varnothing\\
			& & \varnothing & \{7,8\}\\
			& & & \{9,10\}
		\end{psmallmatrix} \quad\xmapsto{\Xi}\quad \vcenter{\hbox{\begin{tikzpicture}[line width=0.8pt,scale=0.4]
					\coordinate (O) at (0,0);
					\draw[thick] (O)--++(1,1)--++(1,1)--++(1,-1)--++(1,0)--++(1,-1);
					\filldraw (O) circle(0.5ex) ++(1,1)circle(0.5ex) ++(1,1)circle(0.5ex) ++(1,-1)circle(0.5ex)
					++(1,0)circle(0.5ex) ++(1,-1)circle(0.5ex);
					\node[above] at (0,0) {\footnotesize $1$};
					\node[above] at (1,1) {\footnotesize $2$};
					\node[above] at (2,2) {\footnotesize $2$};
					\node[above] at (3,1) {\footnotesize $1$};
					\node[above] at (4,1) {\footnotesize $2$};
					\node[above] at (5,0) {\footnotesize $2$};
		\end{tikzpicture}}} \quad = \quad 
		{\setlength{\arraycolsep}{1pt}\begin{array}{ccccccccccc}
				& \rU & & \rU & & \rD & & \rL & & \rD\\
				1 & & 2 & & 2 & & 1 & & 2 & & 2
		\end{array}}\,\,.
	\end{align*}
\end{example}

Now since our goal revolves around nondecreasing \emph{improper} partition matrices, we make the following observation from Definition~\ref{def:MIP}.

\begin{fact}\label{fact:NDIPPM}
	For every nondecreasing improper partition matrix, in each column, only the last nonempty entry (i.e., the nonempty entry with the largest row index) may have an odd cardinality, while all other nonempty entries must have even cardinalities.
\end{fact}

For example, in Table~\ref{tab:CDK-bij-Dyck}, only the second nondecreasing partition matrix is not improper. As another instance, the matrix in \eqref{eq:ndpm-ex1} is improper, while the one in \eqref{eq:ndpm-ex2} is not.

Translating Fact~\ref{fact:NDIPPM} into the setting of weighted Motzkin paths $\cM(\bfx)$ according to our bijective correspondence $\Xi$, we must have an \emph{even} weight $x_i$ whenever the $(i+1)$-th step is a down step, namely, $\rmM_{i+1} = \rD$. Thus, we focus on the following collections of weighted Motzkin paths.

\begin{definition}\label{def:weighted-M}
	For $\bbM'$ a subset of $\bbM_{\ge 0}$, we denote by $[\bbM']$ the collection of restricted weighted Motzkin paths:
	\begin{align*}
		[\bbM'] := \big\{\cM(\bfx):\text{$\cM\in \bbM'$ and $x_i$ is even when $\rmM_{i+1} = \rD$}\big\},
	\end{align*}
	where we write $\cM := (\rmM_1,\rmM_2,\ldots,\rmM_l)$ and $\bfx := (x_0,x_1,\ldots,x_l)$ as before. We further denote by $[\bbM']_n$ the subset of $[\bbM']$ such that the weight $|\cM(\bfx)| = n$.
\end{definition}

The following bijective correspondence plays a central role in our mechanism.

\begin{theorem}\label{th:NDIPPM-core}
	For every $n\ge 1$, there is a bijection $\Xi_n : \NDIPPM_n \to [\bbM_{\ge 0}]_n$ such that
	\begin{align}
		\{\blk,\odd\}\big(B\big) = \{\level^{\x}+1,\odd\}\big(\Xi_n(B)\big)
	\end{align}
	for all $B\in \NDIPPM_n$.
\end{theorem}

\begin{proof}
	Restricting the mapping $\Xi$ in Theorem~\ref{th:Xi} to $\NDIPPM_n$ gives the desired $\Xi_n$.
\end{proof}

\subsection{Analytic proof of Theorem~\ref{th:NDIPPM-Motzkin}}\label{sec:NDIPPM-anal}

Our first proof of Theorem~\ref{th:NDIPPM-Motzkin} relies on a multivariate generating function for Motzkin paths. Let
\begin{align}\label{eq:F-gf}
	F(\mu_1,\mu_2,\mu_3,\nu,\omega) := 1 + \sum_{\cM\in \bbM} \mu_1^{\#_{\cM}(\nearrow)} \mu_2^{\#_{\cM}(\searrow)} \mu_3^{\#_{\cM}(\rightarrow)} \nu^{\#_{\cM}^{\x}(\searrow)} \omega^{\#_{\cM}^{\x}(\rightarrow)}.
\end{align}

\begin{lemma}
	We have
	\begin{align}\label{eq:F-gf-explicit}
		F(\mu_1,\mu_2,\mu_3,\nu,\omega) = \frac{2}{2(1 - \mu_3 \omega) - \nu \big(1 - \mu_3 - \sqrt{(1-\mu_3)^2 - 4\mu_1\mu_2}\big)}.
	\end{align}
\end{lemma}

\begin{proof}
	In view of a result of Flajolet~\cite[p.~129, Theorem~1]{Fla1980} (see also \cite[p.~527, Theorem~A1]{Kra2001}), the generating function $F(1,\mu_2,\mu_3,\nu,\omega)$ can be expressed as a continued fraction:
	\begin{align*}
		F(1,\mu_2,\mu_3,\nu,\omega) = \cfrac{1}{1 - \mu_3 \omega -\cfrac{\mu_2 \nu}{1 - \mu_3 - \cfrac{\mu_2}{1 - \mu_3 - \cfrac{\mu_2}{1 - \mu_3 - \lastddots}}}}.
	\end{align*}
	In particular,
	\begin{align*}
		F(1,\mu_2,\mu_3,1,1) = \frac{1}{1 - \mu_3 - \mu_2 F(1,\mu_2,\mu_3,1,1)},
	\end{align*}
	from which we solve that
	\begin{align*}
		F(1,\mu_2,\mu_3,1,1) = \frac{1 - \mu_3 - \sqrt{(1-\mu_3)^2 - 4\mu_2}}{2\mu_2}.
	\end{align*}
	Here we have discarded the other solution
	\begin{align*}
		\frac{1 - \mu_3 + \sqrt{(1-\mu_3)^2 - 4\mu_2}}{2\mu_2},
	\end{align*}
	because $F(1,\mu_2,\mu_3,1,1)$ should be a power series in $\mathbb{Z}[[\mu_2,\mu_3]]$ according to the definition \eqref{eq:F-gf}, while for the above discarded solution with a plus sign, there is a singularity at $(\mu_2, \mu_3) = (0, 0)$.
	Thus,
	\begin{align*}
		F(1,\mu_2,\mu_3,\nu,\omega) &= \frac{1}{1 - \mu_3 \omega - \mu_2 \nu F(1,\mu_2,\mu_3,1,1)}\\
		&= \frac{2}{2(1 - \mu_3 \omega) - \nu \big(1 - \mu_3 - \sqrt{(1-\mu_3)^2 - 4\mu_2}\big)}.
	\end{align*}
	We then note that on a Motzkin path, up steps are always equinumerous with down steps. Hence, $F(\mu_1,\mu_2,\mu_3,\nu,\omega) = F(1,\mu_1 \mu_2,\mu_3,\nu,\omega)$, which gives the desired expression.
\end{proof}

Now we are in a position to complete the proof of Theorem~\ref{th:NDIPPM-Motzkin}.

\begin{theorem}
	We have
	\begin{align}\label{eq:NDIPPM-gf}
		\sum_{B \in \NDIPPM} y^{\blk(B)} z^{\odd(B)} t^{w(B)} = \frac{y + y z t - 2 y^2 z t - 2 y^2 t^2 - y \sqrt{(1-zt)^2 - 4t^2}}{2 - 2 y - 2 y z t + 2 y^2 z t + 2 y^2 t^2}.
	\end{align}
	In particular, Theorem~\ref{th:NDIPPM-Motzkin} is true.
\end{theorem}

\begin{proof}
	In light of Theorem~\ref{th:NDIPPM-core},
	\begin{align*}
		\sum_{B \in \NDIPPM} y^{\blk(B)} z^{\odd(B)} t^{w(B)} = \sum_{\cM(\bfx)\in [\bbM_{\ge 0}]} y^{\level^{\x}(\cM(\bfx)) + 1} z^{\odd(\cM(\bfx))} t^{|\cM(\bfx)|}.
	\end{align*}
	For the latter, we assume $\cM(\bfx)$ is of length $l$, and note that there is no parity restriction on $x_{i}$ if $\rmM_{i+1} = \rU$ or $\rL$ for $0\le i\le l-1$; in addition, there is no parity restriction on $x_{l}$. Hence, the total contribution to the weight from such $x_i$ is
	\begin{align*}
		(zt + t^2 + zt^3 + t^4 + \cdots)^{\#_{\cM}(\nearrow)+\#_{\cM}(\rightarrow)+1} = \left(\frac{zt+t^2}{1-t^2}\right)^{\#_{\cM}(\nearrow)+\#_{\cM}(\rightarrow)+1},
	\end{align*}
	where we have incorporated the ``$\odd$'' statistic. Meanwhile, $x_{i}$ must be even if $\rmM_{i+1} = \rD$ for $0\le i\le l-1$, thereby giving the contribution
	\begin{align*}
		(t^2 + t^4 + \cdots)^{\#_{\cM}(\searrow)} = \left(\frac{t^2}{1-t^2}\right)^{\#_{\cM}(\searrow)}.
	\end{align*}
	Finally, we recall that $\level^{\x}(\cM(\bfx)) = \#_{\cM}^{\x}(\rightarrow)$. Thus,
	\begin{align*}
		&\sum_{\cM(\bfx)\in [\bbM_{\ge 0}]} y^{\level^{\x}(\cM(\bfx)) + 1} z^{\odd(\cM(\bfx))} t^{|\cM(\bfx)|}\\
		&\qquad = \sum_{\cM\in \bbM_{\ge 0}}y^{\#_{\cM}^{\x}(\rightarrow)+1} \left(\frac{zt+t^2}{1-t^2}\right)^{\#_{\cM}(\nearrow)+\#_{\cM}(\rightarrow)+1} \left(\frac{t^2}{1-t^2}\right)^{\#_{\cM}(\searrow)}\\
		&\qquad = \left(y\cdot \frac{zt+t^2}{1-t^2}\right) F\left(\frac{zt+t^2}{1-t^2}, \frac{t^2}{1-t^2}, \frac{zt+t^2}{1-t^2}, 1, y\right),
	\end{align*}
	yielding \eqref{eq:NDIPPM-gf} by invoking \eqref{eq:F-gf-explicit}. Now to conclude Theorem~\ref{th:NDIPPM-Motzkin}, we notice that
	\begin{align*}
		\sum_{\cM\in \bbM} y^{\comp(\cM)} z^{\level(\cM)} t^{\len(\cM)} = F(t,t,zt,y,y) - 1.
	\end{align*}
	A direct simplification reveals that it is also identical to the right-hand side of \eqref{eq:NDIPPM-gf}, and hence confirms the desired equidistribution in Theorem~\ref{th:NDIPPM-Motzkin}.
\end{proof}

\begin{remark}
	To demonstrate the power of our mechanism involving weighted Motzkin paths, we provide an alternative proof of Theorem~\ref{th:NDPM-Dyck}. First, by virtue of Theorem~\ref{th:Xi},
	\begin{align*}
		\sum_{A \in \NDPM} y^{\blk(A)} z^{\odd(A)} t^{w(A)} = \left(y\cdot \frac{zt+t^2}{1-t^2}\right) F\left(\frac{zt+t^2}{1-t^2}, \frac{zt+t^2}{1-t^2}, \frac{zt+t^2}{1-t^2}, 1, y\right),
	\end{align*}
	yielding the explicit generating function identity
	\begin{align}\label{eq:NDPM-gf-explicit}
		&\sum_{A \in \NDPM} y^{\blk(A)} z^{\odd(A)} t^{w(A)}\notag\\
		&\qquad\qquad = \frac{y + y z t - 2 y^2 z t - 2 y^2 t^2 - y \sqrt{(1 + z t)(1 - 3 z t - 4 t^2)}}{2 - 2 y + 2 z t - 2 y z t + 2 y^2 z t + 2 y^2 t^2}.
	\end{align}
	On the other hand, Dyck paths of semilength $n$ are Motzkin paths of length $2n$ with no level steps. Hence,
	\begin{align*}
		\sum_{\cD\in \bbD} y^{\touch(\cD)} t^{\semilen(\cD)} = F(t^{1/2},t^{1/2},0,y,0) - 1 = \frac{y - 2 y^2 t - y\sqrt{1-4t}}{2 - 2 y + 2 y^2 t},
	\end{align*}
	which equals \eqref{eq:NDPM-gf-explicit} with $z = 1$. Theorem~\ref{th:NDPM-Dyck} therefore follows.
\end{remark}

\subsection{Bijective proof of Theorem~\ref{th:NDIPPM-Motzkin}}\label{sec:NDIPPM-comb}

In view of the key correspondence in Theorem~\ref{th:NDIPPM-core}, we only need to establish the following result to fulfill the desired bijective treatment of Theorem~\ref{th:NDIPPM-Motzkin}.

\begin{theorem}\label{th:NDIPPM-comb}
	For every $n\ge 1$, there is a bijection $\Phi_n : [\bbM_{\ge 0}]_n \to \bbM_n$ such that
	\begin{align}
		\{\level^{\x}+1,\odd\}\big(\cM(\bfx)\big) = \{\comp,\level\} \big(\Phi_n(\cM(\bfx))\big)
	\end{align}
	for all $\cM(\bfx)\in [\bbM_{\ge 0}]_n$.
\end{theorem}

We construct the bijection $\Phi_n$ \emph{inductively} on $n$.

For $n=1$, there is only one weighted Motzkin path $\varnothing(1)$ in $[\bbM_{\ge 0}]_1$, namely, $\cM(\bfx)$ with $\cM=\varnothing$ and $\bfx=(1)$. Now the corresponding Motzkin path in $\bbM_1$ is $\Phi_1(\varnothing(1)) := (\rL)$, the path of only one level step. It is clear that
\begin{align*}
	\{\level^{\x}+1,\odd\}\big(\varnothing(1)\big) = \{\comp,\level\} \big((\rL)\big) = \{1,1\},
\end{align*}
and hence the construction of $\Phi_1$ is completed.

For $n=2$, there are two weighted Motzkin path in $[\bbM_{\ge 0}]_2$, namely, $\varnothing(2)$ and $(\rL)(1,1)$. Now the corresponding Motzkin paths in $\bbM_2$ are $\Phi_2(\varnothing(2)) := (\rU,\rD)$ and $\Phi_2((\rL)(1,1)) := (\rL,\rL)$. It is clear that
\begin{align*}
	\{\level^{\x}+1,\odd\}\big(\varnothing(2)\big) = \{\comp,\level\} \big((\rU,\rD)\big) = \{1,0\},
\end{align*}
and
\begin{align*}
	\{\level^{\x}+1,\odd\}\big((\rL)(1,1)\big) = \{\comp,\level\} \big((\rL,\rL)\big) = \{2,2\}.
\end{align*}
Therefore, $\Phi_2$ is as desired.

Now let us assume that the bijections $\Phi_1,\ldots,\Phi_{n-1}$ are known for a certain $n\ge 3$, and we are about to construct $\Phi_n$.

Define the following subset of $\bbM_{\ge 0}$:
\begin{align*}
	\bbM_{\ge 0}^\dagger := \big\{\cM\in \bbM_{\ge 0}: \level^{\x}(\cM) = 0\big\}.
\end{align*}
A key in our construction is the following bijective mapping.

\begin{lemma}\label{le:key-dagger}
	For every $n\ge 3$, there is a bijection $\Omega_n : [\bbM_{\ge 0}^\dagger]_n \to [\bbM_{\ge 0}]_{n-2}$ with the statistic $\odd$ preserved.
\end{lemma}

\begin{proof}
	Let $\cM(\bfx)\in [\bbM_{\ge 0}^\dagger]_n$ and write the resulting weighted Motzkin path $\Omega_n(\cM(\bfx)) \in [\bbM_{\ge 0}]_{n-2}$ as $\cM'(\bfx')$. If $\cM = \varnothing$ so that $\bfx=(n)$, we define $\cM'(\bfx')$ by $\cM' = \varnothing$ and $\bfx' = (n-2)$. It is clear that $\odd(\cM(\bfx)) = \odd(\cM'(\bfx'))$, as required.
	
	In what follows, we assume that $\cM \ne \varnothing$. Writing $\cM(\bfx)$ in terms of the notation in \eqref{eq:weighted-Motz-word}:
	\begin{align*}
		\begin{array}{ccccccccc}
			& 1 & & \cdots & & \cdots & & l\\
			& \rmM_1 & & \cdots & & \cdots & & \rmM_l\\
			x_0 & & x_1 & & \cdots & & x_{l-1} & & x_l
		\end{array},
	\end{align*}
	we have $\rmM_1 = \rU$ and $\rmM_l = \rD$ since no level steps on $\cM$ are on the $x$-axis. Also, $x_{l-1}$ is even. We have two cases according to the value of $x_{l-1}$.
	
	\begin{enumerate}[label={\textup{(\arabic{*})}}, leftmargin=*, labelsep=5pt, align=left, itemsep=2pt, topsep=2pt, parsep=2pt]
		\item If $x_{l-1}\ge 4$, the resulting weighted Motzkin path $\cM'(\bfx')$ is given by
		\begin{align*}
			\begin{array}{ccccccccc}
				& 1 & & \cdots & & \cdots & & l\\
				& \rmM_1 & & \cdots & & \cdots & & \rmM_l\\
				x_0 & & x_1 & & \cdots & & x_{l-1}-2 & & x_l
			\end{array}.
		\end{align*}
		In particular, $\level^{\x}(\cM'(\bfx')) = 0$.
		
		\item If $x_{l-1}= 2$, we find the \emph{rightmost} up step starting from the $x$-axis; the existence of such an up step is guaranteed because we at least have $\rmM_1 = \rU$ starting from the $x$-axis. Define this step by $\rmM_j = \rU$. We note that this $\rmM_j$ is not the last step. Now $\cM(\bfx)$ can be explicitly written as
		\begin{align*}
			\begin{array}{ccccccccccccccc}
				& 1 & & \cdots & & \cdots & & j & & \cdots & & \cdots & & l\\
				& \rmM_1 & & \cdots & & \cdots & & \rU & & \cdots & & \cdots & & \rD\\
				x_0 & & x_1 & & \cdots & & x_{j-1} & & x_j & & \cdots & & 2 & & x_l
			\end{array}.
		\end{align*}
		Next, we observe that the subpath from $\rmM_j = \rU$ to $\rmM_l = \rD$ never touches the $x$-axis until the ending point. Thus, for $\cM$, after replacing $\rmM_j = \rU$ with a level step $\rL$ and removing $\rmM_l = \rD$, we still arrive at a Motzkin path, denoted by $\cM'$. Finally, we define the resulting weighted Motzkin path $\cM'(\bfx')$ as
		\begin{align*}
			\begin{array}{ccccccccccccccc}
				& 1 & & \cdots & & \cdots & & j & & \cdots & & \cdots & & l-1\\
				& \rmM_1 & & \cdots & & \cdots & & \rL & & \cdots & & \cdots & & \rmM_{l-1}\\
				x_0 & & x_1 & & \cdots & & x_{j-1} & & x_j & & \cdots & & x_{l-2} & & x_l
			\end{array}.
		\end{align*}
		In particular, the $j$-th step on $\cM'$ becomes the \emph{leftmost} level step on the $x$-axis, further implying that $\level^{\x}(\cM'(\bfx')) \ge 1$.
	\end{enumerate}
	It is clear that in both cases, $|\cM'(\bfx')| = |\cM(\bfx)|-2$, while the ``$\odd$'' statistic is preserved. Furthermore, this process is invertible. We therefore arrive at the desired bijection.
\end{proof}

\begin{example}
	We list the following typical examples for $\Omega_9 : [\bbM_{\ge 0}^\dagger]_9 \to [\bbM_{\ge 0}]_{7}$:
	\begin{align*}
		\begin{array}{ccccccccc}
			& \rU & & \rD & & \rU & & \rD\\
			1 & & 2 & & 1 & & {\color{red}4} & & 1
		\end{array} & \quad\mapsto\quad \begin{array}{ccccccccc}
			& \rU & & \rD & & \rU & & \rD\\
			1 & & 2 & & 1 & & {\color{red}2} & & 1
		\end{array},\\
		\\
		\begin{array}{ccccccccc}
			& \rU & & \rD & & {\color{red}\rU} & & {\color{blue}\rD}\\
			1 & & 4 & & 1 & & {\color{blue}2} & & 1
		\end{array} & \quad\mapsto\quad \begin{array}{ccccccc}
			& \rU & & \rD & & {\color{red}\rL}\\
			1 & & 4 & & 1 & & 1
		\end{array},\\
		\\
		\begin{array}{ccccccccc}
			& {\color{red}\rU} & & \rU & & \rD & & {\color{blue}\rD}\\
			1 & & 1 & & 4 & & {\color{blue}2} & & 1
		\end{array} & \quad\mapsto\quad \begin{array}{ccccccc}
			& {\color{red}\rL} & & \rU & & \rD\\
			1 & & 1 & & 4 & & 1
		\end{array},\\
		\\
		\begin{array}{ccccccccc}
			& {\color{red}\rU} & & \rL & & \rL & & {\color{blue}\rD}\\
			1 & & 1 & & 4 & & {\color{blue}2} & & 1
		\end{array} & \quad\mapsto\quad \begin{array}{ccccccc}
			& {\color{red}\rL} & & \rL & & \rL\\
			1 & & 1 & & 4 & & 1
		\end{array}.
	\end{align*}
\end{example}

For the moment, we construct $\Phi_n$ for $n\ge 3$.

\begin{enumerate}[label={\textup{(\arabic{*})}}, leftmargin=*, labelsep=5pt, align=left, itemsep=2pt, topsep=2pt, parsep=2pt]
	\item Assume $\level^{\x}(\cM(\bfx)) = 0$. We define $\Phi_n(\cM(\bfx))$ as
	\begin{align*}
		\Phi_n(\cM(\bfx)) := (\rU, \Phi_{n-2}(\Omega_n(\cM(\bfx))), \rD).
	\end{align*}
	Clearly, there is only one step ending on the $x$-axis, namely, the last step. Hence,
	\begin{align*}
		\comp(\Phi_n(\cM(\bfx))) = 1 = \level^{\x}(\cM(\bfx)) + 1.
	\end{align*}
	Furthermore,
	\begin{align*}
		\level(\Phi_n(\cM(\bfx))) = \level(\Phi_{n-2}(\Omega_n(\cM(\bfx))) = \odd(\Omega_n(\cM(\bfx))) = \odd(\cM(\bfx)),
	\end{align*}
	where we have applied the inductive assumption on $\Phi_{n-2}$ for the second equality and Lemma~\ref{le:key-dagger} for the last equality.
	
	\item Assume $\level^{\x}(\cM(\bfx)) \ge 1$. This time we can uniquely decompose $\cM(\bfx)$ by weighted Motzkin paths $\cM^{(1)}(\bfx^{(1)}),\ldots,\cM^{(k)}(\bfx^{(k)})$ of smaller weight in $[\bbM_{\ge 0}^\dagger]$; here
	\begin{align*}
		k = \level^{\x}(\cM(\bfx)) + 1.
	\end{align*}
	Explicitly, $\cM(\bfx)$ is represented as
	\begin{align*}
		\scalebox{0.9}{%
			$\begin{array}{ccccccccccccc}
				& \rmM_1^{(1)} & \cdots & \rmM_{l^{(1)}}^{(1)} & & \fbox{$\rL$} & \cdots & \fbox{$\rL$} & & \rmM_1^{(k)} & \cdots & \rmM_{l^{(k)}}^{(k)}\\
				x_0^{(1)} & & \cdots & & x_{l^{(1)}}^{(1)} & & \cdots & & x_0^{(k)} & & \cdots & & x_{l^{(k)}}^{(k)}
			\end{array}$
		}.
	\end{align*}
	Here the boxed \fbox{$\rL$} are exactly the level steps on $\cM(\bfx)$ lying on the $x$-axis. Since each $\cM^{(i)}(\bfx^{(i)})$ has a smaller weight than $\cM(\bfx)$, it is true that $|\bfx^{(i)}| < n$ so that the bijection $\Phi_{|\bfx^{(i)}|}$ has been known according to the inductive assumption. We then map each $\cM^{(i)}(\bfx^{(i)})$ to a Motzkin path by $\Phi_{|\bfx^{(i)}|}(\cM^{(i)}(\bfx^{(i)}))$; in particular,
	\begin{align*}
		\{\comp,\level\} \big(\Phi_{|\bfx^{(i)}|}(\cM^{(i)}(\bfx^{(i)}))\big) = \{1, \odd(\cM^{(i)}(\bfx^{(i)}))\},
	\end{align*}
	where for the ``$\comp$'' statistic we have used the fact that $\cM^{(i)}(\bfx^{(i)})\in [\bbM_{\ge 0}^\dagger]$ so that $\level^{\x}(\cM^{(i)}(\bfx^{(i)})) = 0$. Now $\Phi_n(\cM(\bfx))$ is defined as
	\begin{align*}
		\Phi_n(\cM(\bfx)) := (\Phi_{|\bfx^{(1)}|}(\cM^{(1)}(\bfx^{(1)})),\ldots,\Phi_{|\bfx^{(k)}|}(\cM^{(k)}(\bfx^{(k)}))).
	\end{align*}
	We have
	\begin{align*}
		\comp(\Phi_n(\cM(\bfx))) &= \sum_{i=1}^k \comp(\Phi_{|\bfx^{(i)}|}(\cM^{(i)}(\bfx^{(i)}))) = \sum_{i=1}^k 1 = k\\
		& = \level^{\x}(\cM(\bfx)) + 1,
	\end{align*}
	and
	\begin{align*}
		\level(\Phi_n(\cM(\bfx))) &= \sum_{i=1}^k \level(\Phi_{|\bfx^{(i)}|}(\cM^{(i)}(\bfx^{(i)}))) = \sum_{i=1}^k \odd(\cM^{(i)}(\bfx^{(i)}))\\
		& = \odd(\cM(\bfx)),
	\end{align*}
	as desired.
\end{enumerate}

\begin{example}
	We determine $\Phi_3$ completely. There are four weighted Motzkin paths $\cM(\bfx)$ in $[\bbM_{\ge 0}]_3$, namely, $\varnothing(3)$, $(\rL)(1,2)$, $(\rL)(2,1)$, and $(\rL,\rL)(1,1,1)$.
	\begin{itemize}[itemindent=*, leftmargin=*, itemsep=2pt, topsep=2pt, parsep=2pt]
		\item For $\varnothing(3)$, we have $\level^{\x}(\varnothing(3)) = 0$. Now apply $\Omega_3$ to $\varnothing(3)$ and get $\Omega_3(\varnothing(3)) = \varnothing(1)$. Then,
		\begin{align*}
			\Phi_3(\varnothing(3)) = (\rU,\Phi_1(\varnothing(1)),\rD) = (\rU,\rL,\rD).
		\end{align*}
		
		\item For $(\rL)(1,2)$, we have $\level^{\x}((\rL)(1,2)) = 1$. Now decompose it by two weighted Motzkin paths in $[\bbM_{\ge 0}^\dagger]$, namely, $\varnothing(1)$ and $\varnothing(2)$. Then,
		\begin{align*}
			\Phi_3((\rL)(1,2)) = (\Phi_1(\varnothing(1)),\Phi_2(\varnothing(2))) = (\rL,\rU,\rD).
		\end{align*}
		
		\item For $(\rL)(2,1)$, we have $\level^{\x}((\rL)(2,1)) = 1$. Now decompose it by two weighted Motzkin paths in $[\bbM_{\ge 0}^\dagger]$, namely, $\varnothing(2)$ and $\varnothing(1)$. Then,
		\begin{align*}
			\Phi_3((\rL)(2,1)) = (\Phi_2(\varnothing(2)),\Phi_1(\varnothing(1))) = (\rU,\rD,\rL).
		\end{align*}
		
		\item For $(\rL,\rL)(1,1,1)$, we have $\level^{\x}((\rL,\rL)(1,1,1)) = 2$. Now decompose it by three weighted Motzkin paths in $[\bbM_{\ge 0}^\dagger]$, namely, $\varnothing(1)$, $\varnothing(1)$, and $\varnothing(1)$. Then,
		\begin{align*}
			\Phi_3((\rL,\rL)(1,1,1)) = (\Phi_1(\varnothing(1)),\Phi_1(\varnothing(1)),\Phi_1(\varnothing(1))) = (\rL,\rL,\rL).
		\end{align*}
	\end{itemize}
\end{example}

\begin{example}
	We determine $\Phi_7((\rU,\rD,\rL)(2,2,1,2))$. Note that
	\begin{align*}
		\level^{\x}((\rU,\rD,\rL)(2,2,1,2)) = 1.
	\end{align*}
	Now decompose it by two weighted Motzkin paths in $[\bbM_{\ge 0}^\dagger]$, namely, $(\rU,\rD)(2,2,1)$ and $\varnothing(2)$. Then it is necessary to determine $\Phi_5((\rU,\rD)(2,2,1))$ in advance. For this, note that $\level^{\x}((\rU,\rD)(2,2,1)) = 0$. Apply $\Omega_5$ to $(\rU,\rD)(2,2,1)$ and get $\Omega_5((\rU,\rD)(2,2,1)) = (\rL)(2,1)$. Then,
	\begin{align*}
		\Phi_5((\rU,\rD)(2,2,1)) = (\rU, \Phi_3((\rL)(2,1)), \rD) = (\rU, \rU, \rD, \rL, \rD).
	\end{align*}
	Finally,
	\begin{align*}
		\Phi_7((\rU,\rD,\rL)(2,2,1,2)) = (\Phi_5((\rU,\rD)(2,2,1)), \Phi_2(\varnothing(2))) = (\rU, \rU, \rD, \rL, \rD, \rU, \rD).
	\end{align*}
\end{example}

\section{Outlook}\label{sec:outlook}

In Theorem~\ref{thm:q=-1}, our attention is paid to the evaluation of the polynomials $S_n(q) = \sum_{P\in \MatP_{n}} q^{\inv(P)}$ at $q=-1$. Meanwhile, we have computed the exact expressions of $S_n(q)$ for some small $n$:
\begin{align*}
	S_1(q) &= 1,\\
	S_2(q) &= 2,\\
	S_3(q) &= 5 + q,\\
	S_4(q) &= 15 + 7 q + 2 q^2,\\
	S_5(q) &= 53 + 41 q + 20 q^2 + 5 q^3 + q^4,\\
	S_6(q) &= 217 + 240 q + 161 q^2 + 68 q^3 + 24 q^4 + 8 q^5 + 2 q^6,\\
	S_7(q) &= 1014 + 1475 q + 1253 q^2 + 716 q^3 + 334 q^4 + 154 q^5 + 62 q^6 + 22 q^7 + 9 q^8 + q^9,\\
	S_8(q) &= 5335 + 9677 q + 9950 q^2 + 7066 q^3 + 4034 q^4 + 2192 q^5 + 1098 q^6 + 527 q^7 \\
	&\quad + 271 q^8 + 108 q^9 + 40 q^{10} + 18 q^{11} + 4 q^{12}.
\end{align*}
The following questions are then straightforward.

\begin{question}\label{ques:Sn-closed-expr}
	Is there a closed expression for the generating series $\sum_{n\ge 1} S_n(q) t^n$?
\end{question}

\begin{question}
	For each $n\ge 1$, are the coefficients in the polynomial $S_n(q)\in \mathbb{N}[q]$ unimodal? Moreover, do these coefficients satisfy a certain distributional law when $n$ is sufficiently large?
\end{question}

In addition, since partition matrices are equinumerous with inversion sequences (or permutations), one of the referees suggested the following direction for exploration.

\begin{question}
	Is there a ``natural'' statistic for inversion sequences or permutations such that it is equidistributed with the statistic $\inv$ for partition matrices?
\end{question}

Along this line, some further insights arise naturally. Let us call any set of combinatorial objects enumerated by the Fishburn number $\Fis_n$ (resp.~the factorial $n!$) a \emph{Fishburn structure} (resp.~a \emph{factorial structure}). As witnessed by \eqref{eq:S0} and \eqref{eq:S1}, our polynomial $S_n(q)$ interpolates between a Fishburn structure (namely, $\MatF_n$) and a factorial structure (namely, $\MatP_n$). Recall that there are at least four other Fishburn structures that have been well studied, namely, the $(\mathbf{2}+\mathbf{2})$-free unlabeled posets (also known as the interval orders), the ascent sequences, the Fishburn permutations, and the Stoimenow matchings; see \cite{BCDK2010} for undefined terms. More notably, each of these Fishburn structures admits a superset that is a factorial structure. Now the following question seems pertinent.

\begin{question}
	Does there exist a ``natural'' statistic on any of the four remaining factorial structures, such that the refined enumeration according to this statistic gives rise to an alternative interpretation of $S_n(q)$? Or from a different perspective, do any of these other four structures possess natural interweaving $q$-enumeration such that the signed (i.e., $q=-1$) counting leads to further findings?
\end{question}

Lastly, a key concept we introduce in this work is improper partition matrices. As we have discussed in Remarks~\ref{rmk:no-bij} and \ref{rmk:CDK-NDIPPM}, the natural Claesson--Dukes--Kubitzke bijection does not map improper partition matrices to the restricted inversion sequences considered in Theorem~\ref{thm:q=-1}. Examinations on larger improper partition matrices indicate that the image may behave chaotically. Thus, it is meaningful, as pointed out by one of the referees, to consider the following question.

\begin{question}
	Find a characterization (possibly in terms of pattern avoidance) of the inversion sequences produced by improper partition matrices under the action of the Claesson--Dukes--Kubitzke bijection.
\end{question}

Recall also that our proof of Theorem~\ref{th:MIP-I} builds on several known results such as the Frobenius identity \eqref{eq:Frob} and is analytic in nature. It remains an open problem to find a direct statistic-preserving bijection between $\IPPM_n$ and $\bI_n(-,-,=)$. To that end, we make some further observations here, which hopefully would motivate the interested reader to pursue this question. First, let us split both $\IPPM_n$ and $\bI_n(-,-,=)$ into two complementary subsets respectively. For every $n\ge 1$, we decompose $\IPPM_n$ as the \emph{disjoint union} $\IPPM_n^{+}\sqcup \IPPM_n^{-}$, where
\begin{align*}
	\IPPM_n^{+} &:= \{Q\in\IPPM_n:\text{$n$ belongs to a subset in $Q$ of an even cardinality}\},\\
	\IPPM_n^{-} &:= \{Q\in\IPPM_n:\text{$n$ belongs to a subset in $Q$ of an odd cardinality}\}.
\end{align*}
In a similar vein, we split $\bI_n(-,-,=)$ into two disjoint subsets $\bI_n^{+}(-,-,=)$ and $\bI_n^{-}(-,-,=)$, where
\begin{align*}
	\bI_n^{+}(-,-,=) &:= \{e\in\bI_n(-,-,=): e_{n-1}=e_n\},\\
	\bI_n^{-}(-,-,=) &:= \{e\in\bI_n(-,-,=): e_{n-1}\neq e_n\}.
\end{align*}
We observe the following correlations between the $+/-$ subsets.

\begin{lemma}\label{lem:even=odd}
	For every $n\ge 1$, there exist bijections
	\begin{align*}
		\phi_n:\,\IPPM_n^- \to \IPPM_{n+1}^+,
	\end{align*}
	and
	\begin{align*}
		\psi_n:\,\bI_n^-(-,-,=) \to \bI_{n+1}^+(-,-,=),
	\end{align*}
	such that for every matrix $Q\in\IPPM_n^-$, we have
	\begin{align}
		\sw(\phi_n(Q))=\sw(Q),\label{eq:sw e=o}
	\end{align}
	while for every inversion sequence $e\in\bI_n^-(-,-,=)$, we have
	\begin{align}
		\dist(\psi_n(e))=\dist(e).\label{eq:dist e=o}
	\end{align}
\end{lemma}

\begin{proof}
	Given a matrix $Q\in \IPPM_n^-$, we can append $n+1$ to the subset containing $n$ to get a matrix $\widetilde{Q}\in\IPPM_{n+1}^+$, which we take to be the image $\phi_n(Q)$. Conversely, given a matrix $\widetilde{Q}\in\IPPM_{n+1}^+$, we can remove its maximal letter $n+1$ to recover its preimage $Q\in\IPPM_n^-$ under $\phi_n$. These two operations are clearly inverses of each other, and \eqref{eq:sw e=o} follows from the fact that $\lceil\frac{m}{2}\rceil=\lceil\frac{m+1}{2}\rceil$ for any odd integer $m$.
	
	To construct $\psi_n$, we note that for any inversion sequence $e\in\bI_n^-(-,-,=)$, its last entry $e_n$ is not repeated so appending another copy of $e_n$ to its end produces an inversion sequence $\tilde{e}\in\bI_{n+1}^+(-,-,=)$, which we take to be the image $\psi_n(e)$. Conversely, removing the last entry of an inversion sequence $\tilde{e}\in\bI_{n+1}^+(-,-,=)$ does recover its preimage $e\in \bI_n^-(-,-,=)$ under $\psi_n$. Defined in this way, $\psi_n$ is seen to be a bijection and the relation \eqref{eq:dist e=o} follows directly.
\end{proof}

Combining Lemma~\ref{lem:even=odd} with the two decompositions presented earlier, we obtain the following relations between the bivariate generating functions:
\begin{align*}
	\sum_{n\ge 1}t^n\sum_{Q\in\IPPM_n}z^{\sw(Q)} = (1+t)\sum_{n\ge 1}t^n\sum_{Q\in\IPPM_n^-}z^{\sw(Q)},
\end{align*}
and
\begin{align*}
	\sum_{n\ge 1}t^n\sum_{e\in\bI_n(-,-,=)}z^{\dist(e)} = (1+t)\sum_{n\ge 1}t^n\sum_{e\in\bI_n^-(-,-,=)}z^{\dist(e)}.
\end{align*}
Consequently, our inquiry for a direct bijective proof of Theorem~\ref{th:MIP-I} can be equivalently rephrased as follows.

\begin{question}
	For every $n\ge 1$, find a direct bijection $\rho_n:\IPPM_n^- \to \bI_n^-(-,-,=)$ such that
	\begin{align*}
		\sw(Q) = \dist(\rho_n(Q))
	\end{align*}
	for all matrices $Q\in\IPPM_n^-$.
\end{question}

\subsection*{Acknowledgements}

Shane Chern was fully supported by the Austrian Science Fund (grant no.~10.55776/F1002). Shishuo Fu was partially supported by the Fundamental Research Funds for the Central Universities (grant no.~2025CDJ-IAISYB-008) and the National Natural Science Foundation of China (grant no.~12171059).

\bibliographystyle{amsplain}

\end{document}